\documentclass[12pt,a4paper]{amsart}

\makeatletter
 \def\@textbottom{\vskip \z@ \@plus 582pt}
 \let\@texttop\relax
\makeatother

\usepackage{amsmath,amsfonts,amssymb,mathrsfs}
\usepackage[english]{babel}
\usepackage[utf8]{inputenc}
\usepackage[T1]{fontenc}
\usepackage{graphicx}
\usepackage{tabularx}
\usepackage{hyperref}
\usepackage{tikz}
\usetikzlibrary{arrows}
\usepackage{fancyhdr}

\allowdisplaybreaks

\newcommand{\nocontentsline}[3]{}
\newcommand{\tocless}[2]{\bgroup\let\addcontentsline=\nocontentsline#1{#2}\egroup}
\DeclareUnicodeCharacter{0177}{\^u}

\newtheorem{theorem}{Theorem}[section]

\newtheorem{lemma}[theorem]{Lemma}
\newtheorem{definition}[theorem]{Definition}

\newtheorem{proposition}[theorem]{Proposition}
\newtheorem{corollaire}[theorem]{Corollary}

\newtheorem{remarque}[theorem]{Remark}
\newtheorem{example}[theorem]{Example}

\newcommand{\R}{\mathbb{R}}

\newcommand{\Z}{\mathbb{Z}}
\renewcommand{\P}{\mathbb{P}}
\newcommand{\Q}{\mathbb{Q}}

\newcommand{\C}{\mathbb{C}}
\newcommand{\N}{\mathbb{N}}

\begin{document}

\title[]{On the equidistribution of some Hodge loci} %\footnote{}
\author{Salim Tayou }
\date\today
\thanks{\textit{École Normale Supérieure, 45 rue d'Ulm, 75005 Paris, France}}
\thanks{\textit{Laboratoire de Mathématiques d'Orsay, Université Paris-Sud, Orsay, France}}
\thanks{E-mail adress: salim.tayou@math.u-psud.fr}

\begin{abstract}
We prove the equidistribution of the Hodge locus for certain non-isotrivial, polarized variations of Hodge structure of weight $2$ with $h^{2,0}=1$ over complex, quasi-projective curves. Given some norm condition, we also give an asymptotic on the growth of the Hodge locus. In particular, this implies the equidistribution of elliptic fibrations in quasi-polarized, non-isotrivial families of K3 surfaces.
\end{abstract}
					
\maketitle

\tableofcontents

\section{Introduction}	

Let $S$ be a complex, quasi-projective curve and let $\{\mathbb{V}_{\Z},\mathcal{F}^{\bullet}\mathcal{V},Q\}$ be an integral, polarized variation of Hodge  structure  of weight $2$ over $S$ with $h^{2,0}=1$. We assume that the bilinear form associated $Q$ is even of signature $(2,b)$. For $s\in S$, let $\rho(\mathbb{V}_{\Z,s})$ be the {\it Picard number} of $\mathbb{V}_{\Z,s}$, that is, the rank of the group of  integral $(1,1)$-classes in $\mathbb{V}_{\Z,s}$. Let $M$ be the minimum value of the integers $\rho(\mathbb{V}_{\Z,s})$ for $s\in S$. It is a classical result of Green \cite[Prop. 17.20]{voisin} and Oguiso \cite{oguiso} that the Noether-Lefschetz locus 
$$\mathrm{NL}(\mathbb{V}_{\Z})=\{s\in S,\, \rho(\mathbb{V}_{\Z,s})> M \}$$ is a countable dense subset of $S$, when the variation of Hodge structure is non-trivial. A weaker result obtained in \cite{bkpsb} assumes the base $S$ to be projective. One might then ask how the set of points where the Picard rank jumps distributes inside $S$. The goal of this paper is to investigate quantitative statements on the distribution of the Hodge locus. 
\bigskip

Say that the polarized variation of Hodge structure $\{\mathbb{V}_{\Z},\mathcal{F}^{\bullet}\mathcal{V},Q\}$ is simple if there is no polarized sub-variation of Hodge structure $\{\mathbb{V'}_{\Z}$, $\mathcal{F}^{\bullet}\mathcal{V'},Q\}$ such that $\mathbb{V}_{\Z}/\mathbb{V'}_{\Z}$ is non-zero and torsion free. In fact, starting from an arbitrary polarized variation of Hodge structure $\{\mathbb{V}_{\Z},\mathcal{F}^{\bullet}\mathcal{V}$, $Q\}$ with $h^{2,0}=1$, the minimal sub-variation $\{\mathbb{V'}_{\Z},\mathcal{F}^{\bullet}\mathcal{V'},Q\}$ for which $\mathbb{V}_{\Z}/\mathbb{V'}_{\Z}$ is torsion free is simple. Its orthogonal with respect to $Q$ is also an integral sub-variation of Hodge structure by Deligne and Schmid's semi-simplicity theorem \cite{deligne,schmid} and which is purely of type $(1,1)$, thus it is trivial, up to taking a finite étale cover of $S$.

Say also that $\{\mathbb{V}_{\Z},\mathcal{F}^{\bullet}\mathcal{V},Q\}$ is non-trivial if the line bundle $\mathcal{F}^{2}\mathcal{V}$ is not isotrivial. By a result of Griffiths \cite[Chapter II]{topics} the first Chern class $\omega$ of $\mathcal{F}^{2}\mathcal{V}$ is positive definite, and the integration with respect to $\omega$ defines a finite measure $\mu$ on $S$. 
\bigskip

Assume that for each $s\in S$, the  lattice $(\mathbb{V}_{\Z,s},Q)$ is isomorphic to an even quadratic lattice $(V,(\,.\,))$ of signature $(2,b)$ with $b\geq 3$ and $Q(x)=\frac{(x.x)}{2}$ for all $x\in V$. %We make the following assumption: the cardinal of the discriminant lattice $V^{\vee}/V$ is square free. Actually a weaker  but rather technical assumption is needed to obtain the main result (see \ref{technical}) but we prefer to give this condition to have a concise statement.
\bigskip

Let $\mathbb{V}_{\Z}^{\vee}\subset \mathbb{V}_{\Q}$ be the dual local system to $\mathbb{V}_{\Z}$ with respect to $Q$, i.e the fiber $\mathbb{V}_{\Z,s}^{\vee}$ at each point $s\in S$ is equal to $$\{x\in \mathbb{V}_{\Q,s},\,\forall y\in \mathbb{V}_{\Z,s},\, (x.y)\in \Z\}.$$ The fibers of the local system $\mathbb{V}_{\Z}^{\vee}/\mathbb{V}_{\Z}$ are isomorphic to the finite group $V^{\vee}/V$. For $s\in S$ and $\lambda\in\mathbb{V}^{\vee}_{\Z,s}$, there exists $\gamma\in V^{\vee}/V$ such that $\lambda\in \gamma+V$ with the previous identification, and therefore $Q(\lambda)\in Q(\gamma)+\Z$. In general  we have  $Q(\lambda)\in \cup_{\gamma\in V^{\vee}/V}(Q(\gamma)+\Z)$.   

\bigskip 
If  $H$  is a subgroup of $V^{\vee}/V$ we let  $H^{\bot}$ be the orthogonal of $H$ in $V^{\vee}/V$ with respect to the reduction of the form $(\,.\,)$ valued in $\Q/\Z$. The main result of this paper is the following theorem.
\begin{theorem}\label{main}
Let $\{\mathbb{V}_{\Z},\mathcal{F}^{\bullet}\mathcal{V},Q\}$ be a non-trivial, polarized, simple variation of Hodge structure of weight $2$ over a quasi-projective curve $S$ with $h^{2,0}=1$. Assume that the quadratic lattice $(V,(\,,\,))$ is even and that the local system $\mathbb{V}_{\Z}^{\vee}/\mathbb{V}_{\Z}$ is trivial. Let $H$ be a maximal totally isotropic subgroup of $V^{\vee}/V$. 
Let $\mu$ be the measure defined by integrating the first Chern class of $\mathcal{F}^{2}\mathcal{V}$ and let $\gamma\in H^{\bot}\subset V^{\vee}/V$. Set $A_{\gamma}=\{-Q(x+\gamma),\,x\in V \}$. Then
\begin{itemize}
\item[(i)] For $n\in \Q_{>0}$, the number $N(\gamma,n)$ of points $s\in S$ (counted with multiplicity) for which there exists a $(1,1)$-element $x$ in $\mathbb{V}_{\Z,s}^{^\vee}$ of class $\gamma$ in $\mathbb{V}_{\Z,s}^{\vee}/ \mathbb{V}_{\Z,s}$ and $(x.x)=-2n$ is equal to zero if $n\notin A_{\gamma}$, otherwise it satisfies  

$$N(\gamma,n)\sim\mu(S)\frac{(2\pi)^{1+\frac{b}{2}}.n^{\frac{b}{2}}}{\sqrt{|V^{\vee}/V|}\Gamma(1+\frac{b}{2})}.\prod_{p}\mu_p(\gamma,n,V)$$
as $n$ tends to infinity along $A_{\gamma}$, where  $$\prod_{p<\infty}\mu_p(\gamma,n,V)\asymp 1.$$ %and $$u=\max(\frac{2+b}{4}+\epsilon,\frac{b}{2}-\epsilon).$$ 
If $S$ is projective, then the error term is  $O_{\epsilon}(n^{\frac{2+b}{4}+\epsilon})$ for every $\epsilon>0$.
\item[(ii)] The set of such points equidistributes in $S$ with respect to $\mu$.
\end{itemize}
\end{theorem}

We refer to Example \ref{eisenstein} for the definition of the factors $\mu_p(\gamma,n,V)$. The product $\prod_{p<\infty}\mu_p(\gamma,n,V)$ is called the {\it singular series}. The Hodge locus has a schematic structure (see \cite[Chapter 17]{voisin}). The  multiplicity of a point evoked in Theorem \ref{main} is the multiplicity in this schematic sense.
\bigskip

This number can also be seen as the multiplicity of intersection of $S$ with  special divisors, the so-called { \it Heegner divisors}, in the moduli space of Hodge structure of K3 type over $V$. This moduli space is in fact a Shimura variety of orthogonal type. As a part of their study of the André--Oort conjecture  \cite{andreoort}, Clozel and Ullmo  proved in \cite{clozelullmo} that the Heegner divisors are equidistributed with respect to the measure induced by the Bergman metric (\cite[Chap VIII]{helgason}). The latter is simply given, up to an absolute constant, by integrating the top power of the first Chern class of the Hodge bundle. What we prove here is the equidistribution of the intersection of the Heegner divisors with any fixed quasi-projective curve as above with respect to the measure given by integration the first Chern class of the Hodge bundle. 
 %Recall that the André--Oort conjecture asks whether a subvariety of a Shimura variety  admitting a Zariski dense subset of special points is itself special.
% The situation here is quite different since the curve is not necessarily special, but still its intersection with Heegner divisors equidistributes in $S$.    
\bigskip

\begin{corollaire}
Let $\{\mathbb{V}_{\Z},\mathcal{F}^{\bullet}\mathcal{V},Q\}$ be a non-trivial, polarized, simple variation of Hodge structure of weight $2$ over a quasi-projective curve $S$ and $\mu$ the measure on $S$ defined by integrating the first Chern class of $\mathcal{F}^{2}\mathcal{V}$.  Then the set of points $s\in S$ for which there exists a $(1,1)$-class $x$ in $\mathbb{V}_{\Z,s}$ such that $(x.x)=-2n$ are equidistributed with respect to $\mu$ along positive integers $n$ in the infinite set $\{-Q(x),\, x\in V\}$.
\end{corollaire}

Indeed, we can always take a finite étale cover $\tilde{S}$ of $S$  for which the pullback of the local system $\mathbb{V}_{\Z}^{\vee}/ \mathbb{V}_{\Z}$ is trivial and apply Theorem \ref{main} to $\tilde{S}$ which imply the corollary.  

In particular, we deduce various equidistribution results for points  in some $1$-parameter families of complex varieties. 
\begin{corollaire}\label{k3}
Let $\Lambda_{K_3}$ be the K3 lattice and $P\subset \Lambda_{K_3}$ a primitive Lorentzian anisotropic sublattice of rank $\rho\leq 4$. Let $\mathcal{X}\xrightarrow{\pi} S$ be a non-isotrivial family of K3 surfaces with generic Picard group equal to $P$ over a quasi-projective curve $S$  and let $\{R^{2}\pi_{*}\underline{\Z}_{\mathcal{X}},\mathcal{F}^{\bullet}\mathcal{H},Q\}$ be the induced variation of Hodge structure on $S$. Let $\mu$ be the measure induced by integrating the first Chern class of $\mathcal{F}^{2}\mathcal{H}$. Fix $H\subset P^{\vee}/P$ a maximal isotropic group,  $\gamma\in H^{\bot}$ and let $A_{\gamma}=\{Q(x+\gamma),\, x\in P\}$.
\begin{itemize}
\item[(i)] For $n\in\Q_{>0}$, the number $N(\gamma,n)$ of points $s\in S$ (counted with multiplicity) for which  $\mathcal{X}_{s}$ admits a parabolic line bundle of type $(\gamma,n)$ is zero if $n\notin A_{\gamma}$, otherwise it satisfies $$N(\gamma,n)\sim\mu(S)\frac{(2\pi)^{\frac{22-\rho}{2}}.n^{10-\frac{\rho}{2}}}{\sqrt{|P^{\vee}/P|}\Gamma(\frac{22-\rho}{2})}.\prod_{p<\infty}\mu_p(n,\gamma,V)$$
 as $n$ tends to infinity in $A_{\gamma}$, and where $V=P^{\bot}$.% where $$u=\max(\frac{2+b}{4}+\epsilon,\frac{b}{2}-\epsilon).$$ 

If $S$ is projective, then the error term is  $O_{\epsilon}(n^{\frac{2+b}{4}+\epsilon})$ for every $\epsilon>0$.
\item[(ii)] The previous set is equidistributed in $S$ with respect to $\mu$.
\item[(iii)] If $P^{\vee}/P$ has no non-trivial isotropic subgroup, then the set of points $s\in S$ (counted with multiplicity) for which $\mathcal{X}_{s}$ admits an elliptic fibration of norm less than $n$ is equidistributed with respect to $\mu$ as $n$ tends to infinity. 
\end{itemize}
\end{corollaire}

%\begin{corollaire}
%Let $\mathcal{X}\xrightarrow{\pi} S$ be a non-isotrivial, quasi-polarized family of K3 surfaces of degree $2d$  over a quasi-projective curve $S$ and $\{R^{2}\pi_{*}\underline{\Z}_{\mathcal{X}},\mathcal{F}^{\bullet}\mathcal{H},Q\}$ the induced variation of Hodge structure on $S$. Let $\mu$ be the measure induced by integrating the first Chern class of $\mathcal{F}^{2}\mathcal{H}$. 
%\begin{itemize}
%\item[(i)] For $\epsilon>0$, the number $N(\gamma,n)$ of points $s\in S$ (counted with multiplicity) for which  $\mathcal{X}_{s}$ admits an elliptic fibration of degree $n$ satisfies $$N(\gamma,n)\sim\mu(S)\frac{(2\pi)^{\frac{21}{2}}.n^{19}}{\sqrt{2d}\Gamma(\frac{21}{2})}.\prod_{p<\infty}\mu_p(n,2d)$$
% as $n$ tends to infinity in $\N$.
%\item[(ii)] The previous set is equidistributed in $S$ with respect to $\mu$.
%\end{itemize}
%\end{corollaire}

For the definition of a parabolic line bundle of type $(\gamma,n)$ and the norm of an elliptic fibration, we refer to Definition\ref{degree}. If the Lorentzian sublattice $P$ is generated by a single element, the corollary says that the number of elliptic surfaces of norm less than $n^2$ (or volume less than $n$) in a generic family of quasi-polarized K3 surfaces "grows like" $n^{20}$. In the case of twistor families of K3 surfaces, an analogous result was shown by Simion Filip in \cite{filip} and an improvement of the error term was given by Bergeron and Matheus in \cite{bergeron}. The main term there grows also like $n^{20}$, and Filip works with the full K3 lattice $\Lambda_{K3}$. His method  is different from ours, although it was the starting point of this paper. Notice also the analogy between the coefficient of the main term in our case and in Filip's case. Indeed, due to the Siegel mass formula (see \cite{eskinsarnak}), the product $\prod_{p<\infty}\mu_p(n,\gamma,P)$  can be expressed as a sum of volumes of some homogeneous spaces (compare to Filip's formula 3.1.6 in \cite{filip}). 
There is also a generalization of the previous corollary which concerns families of hyperkähler manifolds over a quasi-projective curve which we discuss of the section 4.3.

%Conjecturally, a nef parabolic $(1,1)$-class on a hyperkähler variety is semi-ample. This is the statement of the Strominger-Yau-Zaslow (SYZ) conjecture (\cite{syz}). Admitting the latter, corollary \ref{hyperkähler} would yield  equidistribution of Lagrangian classes, i.e, semi-ample parabolic line bundles.

\bigskip
%For smaller values of $b$ and which arre note covered by \ref{main}, the analogous statement for equiditribution of Heegner points in m the modular surface $Y(1)(\C)$ is the main result of Clozel and Ullmo in \cite, while the equidistribution of Hirzerbruch-Zagier divisors in Hilbert modular surfaces 
There are several  arithmetic statements which shed light on the arithmetic analogues of the above results, the curve $S$ being replaced by an open subset of the spectrum of the ring of integers of a number field. A result by Charles \cite{charles1} shows that the set of primes  where the reduction of two elliptic curves defined over a number field  are geometrically isogenous is infinite. More recently, Shankar and Tang \cite{tang} proved by using similar techniques that, given a simple abelian surface defined over a number field and which has real multiplication, there are infinitely many places where  its reduction is not absolutely simple. 
\bigskip

\subsection{Outline of the proof} Let us now sketch the proof of Theorem \ref{main}. Let $D_{V}$ be the period domain associated to the quadratic lattice $(V,Q)$, namely the complex analytic variety defined by  $$D_{V}=\{x\in\P(V_{\C}),\,(x.x)=0,\,(x.\overline{x})>0\}.$$
Let $\Gamma_{V}$ be the stable orthogonal group of $V$ defined by $$\Gamma_{V}:=\mathrm{Ker}\left(\mathrm{O}(V)\rightarrow \mathrm{O}(V^{\vee}/V)\right),$$ 
where $$V^{\vee}:=\{x\in V_{\Q},\, \forall y\in V,\, (x.y)\in\Z\}$$ denotes as before the dual lattice of $V$. By Baily and Borel \cite{bailyborel}, the complex analytic quotient $\Gamma_{V}\backslash D_{V}$ can be endowed with a natural  structure of quasi-projective variety called an {\it orthogonal modular variety}. It is the structure that we consider in the whole text. Let $\mathcal{L}$ denote the Hodge bundle on $\Gamma_{V}\backslash D_{V}$ and let $\omega$ be its first Chern class. Recall that $\mathcal{L}$ is an ample line bundle \cite{bailyborel}.
\bigskip

For $\gamma\in V^{\vee}/V$ and $n\in -Q(\gamma)+\Z$ with $n>0$, let $\mathcal{Z}(\gamma,-n)$ denote the associated Heegner divisor in $\Gamma_{V}\backslash D_{V}$ which  parametrizes Hodge structures on $V$ for which there exists a rational Hodge class  $\lambda\in\gamma+V$ with $(\lambda.\lambda)=-2n$ (see Section \ref{heegnerdivisors} for the precise definition).  Let $\{S,\mathbb{V}_{\Z},\mathcal{F}^{\bullet}\mathcal{V},Q)\}$ be given as in Theorem \ref{main}. Since the local system $\mathbb{V}_{\Z}^{\vee}/\mathbb{V}_{\Z}$ is trivial, we have a corresponding holomorphic period map $$\rho: S \rightarrow\Gamma_{V}\backslash D_{V} .$$ This map  is in fact algebraic by Borel \cite{borelmetric}\footnote{In fact, in \cite{borelmetric} the theorem is stated for smooth quotients but see \cite[Remark 4.2]{huybrechts} for how one can reduce to this case.}. The pullback of the Hodge bundle $\mathcal{L}$ along $\rho$ is equal to $\mathcal{F}^{2}\mathcal{V}$. The idea of the proof is to  obtain a global estimate of the cardinality (with multiplicity) of the set $$\{s\in S,\, \exists\, \lambda \in \gamma+\mathbb{V}_{\Z,s},\, (\lambda.\lambda)=-2n\}.$$ To this end, following ideas of Maulik \cite{maulik}, we use linear dependence relations between Heegner divisors to get an upper bound. These relations follow from Borcherds' construction in \cite{borcherdszagier} of a modular form on the Picard group of the orthogonal modular variety $\Gamma_{V}\backslash D_{V}$. Then we extend those relations to a suitable toroidal compactification of $\Gamma_{V}\backslash D_{V}$. It is at this level that we need a restriction on $\gamma$ being in $H^{\bot}$ for a maximal isotropic subgroup $H$ of $V^{\vee}/V$, since for arbitrary $\gamma$,  we don't know how to control the intersection of $S$ with the boundary divisor of the given toroidal compactification. 
\bigskip

To obtain a lower bound, we construct a suitable fibration over every small enough simply connected open subset $\Delta\subset S$. Then following ideas of Green (see \cite[Chap.17]{voisin}), we  obtain  a map to the homogeneous space $A_{0}=\{x\in V_{\R},\, Q(x)=-1\}.$ This map turns out to be, outside a Lebesgue negligible analytic subset, a local diffeomorphism. We use then a result of equidistribution of integral points on $A_0$  proven by Eskin--Oh in a more general context in \cite[Th.1.2]{eskinoh}. The proof of the latter relies on results from ergodic theory, namely the ergodicity of unipotent flows, which is also an important ingredient in the proof of the main result of \cite{clozelullmo}.

\subsection{Outline of the paper} In section 2 we recall the construction of the Borcherds' modular form and its implications on the linear dependence relations between Heegner divisors following ideas of Maulik in \cite{maulik}. We then explain how to extend those relations  to the toroidal  compactification of $\Gamma_{V}\backslash D_{V}$ determined by the perfect cone decomposition following the work of Peterson in  \cite{peterson}. This will allow us, under some mild assumptions, to give global estimates on the growth of the Hodge locus in a curve. We conjecture that these estimates still hold without those assumptions. In section 3 we construct a fibration in spheres over the small open subsets of $S$ which, combined with equidistribution results of Eskin and Oh \cite{eskinoh}, allow to deduce a lower estimate on the cardinality of Hodge locus. In section 4 we explain how one can reduce to the case where the group $V^{\vee}/V$ has no non-trivial isotropic subgroup and then prove the result in this case. The end of the section is devoted to prove corollary \ref{k3}. 

\subsection{Acknowledgements}  I am very grateful to François Charles for introducing me to this subject, for the many discussions we had  and for his enlightening guidance. I wish also to thank Quentin Guignard and Lucia Mocz for their careful reading of an earlier version of this paper. I have benefited from many useful conversations with Yohan Brunebarbe, Gaëtan Chenevier and Étienne Fouvry. Special thanks go to the referee who helped improving the exposition and reducing inaccuracies.
    
This project has received funding from the European Research Council (ERC) under the European Union’s Horizon 2020 research and innovation programme (grant agreement No 715747).
\subsection{Notations} If $\epsilon>0$ $f,g_{\epsilon},h:\N\rightarrow \R$ are  real functions and $g_{\epsilon}$ does not vanish, then: 
\begin{enumerate}
\item $f=O_{\epsilon}(g)$ if there exists an integer $n_\epsilon\in \N$, a positive constant $C_{\epsilon}>0$ such that $$\forall n\geq n_\epsilon,\, |f(n)|\leq C_{\epsilon}|g_{\epsilon}(n)|.$$
\item $f\asymp h$ if and $f=O(h)$ and $h=O(f)$. 
\end{enumerate}
\section{The Weil representation and modular forms}\label{weil1}
\subsection{General setting}\label{generalsetting}
We recall in this section some results about Weil representations and certain vector-valued modular forms associated to quadratic lattices. Our main references are  \cite{borcherds} and \cite{borcherdszagier}.

Let $\mathrm{Mp}_{2}(\R)$ be the metaplectic cover of $\mathrm{SL}_{2}(\R)$: the elements of this group consist of pairs $(M,\phi)$, where $$M=\begin{pmatrix} a&b\\c&d\\ \end{pmatrix}\in \mathrm{SL}_{2}(\R)$$ and $\phi$ is a holomorphic function on the Poincaré upper half-plane $\mathbb{H}$ such  that $\phi(\tau)^2=c\tau +d$, $\tau\in \mathbb{H}$. The group structure is defined by 
$$(M_{1},\phi_1).(M_{2},\phi_{2})=(M_{1}M_{2},\tau\mapsto \phi_{1}(M_{2}.\tau)\phi_{2}(\tau)),$$
for $(M_{1},\phi_1), (M_{2},\phi_{2})\in\mathrm{Mp}_{2}(\R) $, where $M_{2}.\tau$ stands for the usual action of $\mathrm{SL}_{2}(\R)$ on $\mathbb{H}$ given by fractional linear transformations.

The map $\mathrm{Mp}_{2}(\R)\rightarrow \mathrm{SL}_{2}(\R)$  given by $(M,\phi)\mapsto M$ is a double cover of $\mathrm{SL}_{2}(\R)$. Denote by $\mathrm{Mp}_{2}(\Z)$ the inverse image of $\mathrm{SL}_{2}(\Z)$ under this map. 
It is well known (see \cite[P.78]{cours}) that $\mathrm{Mp}_{2}(\Z)$ is generated by the elements: 

$$T=\left(\begin{pmatrix}1&1\\0&1\end{pmatrix},1 \right) \quad \textrm{and}\quad S=\left(\begin{pmatrix}0&-1\\1&0\end{pmatrix},\tau\mapsto\sqrt{\tau} \right).$$
Let $\rho: \mathrm{Mp}_{2}(\Z)\rightarrow \mathrm{GL}(V)$ be a finite-dimensional complex  representation of $\mathrm{Mp}_{2}(\Z)$ that factors through a finite quotient of $\mathrm{Mp}_{2}(\Z)$ and let $k\in \frac{1}{2}\Z$. The group $\mathrm{Mp_{2}}(\Z)$ has a right action on the space of functions $f:\mathbb{H}\rightarrow V$ given by 
\begin{eqnarray}\label{action}
\left(f.(M,\phi)_{k}\right)(\tau)=\phi(\tau)^{-2k}\rho(M,\phi)^{-1}f(M.\tau).
\end{eqnarray}
\bigskip

Fix an eigenbasis $(v_{\gamma})_{\gamma\in I}$ of $V$ with respect to the action of $T$. 
A holomorphic function $f:\mathbb{H}\rightarrow V$ which is invariant under the action of $T$ has a Fourier expansion 

\begin{eqnarray}\label{development}
f(\tau)=\sum_{\gamma\in I}\sum_{n\in \Q}c(\gamma,n)e^{2i\pi n\tau}v_{\gamma}.
\end{eqnarray}
For $(\gamma,n)\in I\times \Q$, the coefficient $c(\gamma,n)$ is non-zero only if  $e^{2i\pi n}$ is the eigenvalue of $T$ acting on $v_{\gamma}$. The function $f$ is said to be holomorphic at infinity if $c(\gamma,n)=0$ for all $n<0$ and $\gamma \in I$.
\begin{definition}
A holomorphic function $f:\mathbb{H}\rightarrow V$ is a \textbf{\textit{modular form of weight $k$ and type $\rho$}}, if it satisfies the following conditions:
\begin{itemize}
\item[(i)] $f$ is invariant under the action (\ref{action}) of $\mathrm{Mp}_{2}(\Z)$.
\item[(ii)] $f$ is holomorphic at infinity.
\end{itemize} 
Moreover, if $c(\gamma,0)=0$ for all $\gamma\in I$ in the formula (\ref{development}), we say that $f$ is a cusp form. 
\end{definition}
Let $M_{k}(\rho)$ denote the $\C$-vector space of modular forms of weight $k$ and type $\rho$, and let $S_{k}(\rho)$ be the subspace of cusp forms. Both $M_{k}(\rho)$ and $S_{k}(\rho)$ are finite-dimensional vector spaces over $\C$ (see \cite[Section 2]{borcherdszagier}). 
%A particular example of a representation as above can be constructed as follows.
\bigskip

Let $(V,Q)$ be an even lattice of signature $(b^+,b^-)$ with  the underlying non-degenerate symmetric bilinear form denoted by  $(\,.\,)$ and such that   $Q(x)=\frac{(x.x)}{2}$ for $x\in V$. Let  $V^{\vee}$ be the dual lattice of $V$. 
% defined by $$L^{\vee}:=\{x\in L\otimes_{\Z}\Q,\, \forall y\in L\, (x,y)\in \Z \}.$$ The quotient $V^{\vee}/V$ is a finite group of order equal to the absolute value of the discriminant of $(\, ,\,)$. The reduction modulo $1$ of $Q$ is a $\Q/\Z$-valued quadratic form on $V^{\vee}/V$, whose associated  $\Q/\Z$-valued bilinear form is the reduction modulo $1$ of the bilinear form $(\, ,\,)$ on $L^{\vee}$.    	
We can associate to the quadratic lattice $(V,Q)$ a representation $\rho_{V}$ of the metaplectic group $\mathrm{Mp}_{2}(\mathbb{Z})$ whose underlying vector space is $\mathbb{C}[V^{\vee}/V]$. For this, it is enough to specify the action of $S$ and $T$ on a basis $(v_{\gamma})_{\gamma \in V^{\vee}/V}$ of $\mathbb{C}[V^{\vee}/V]$ as follows :

\begin{align}\label{weil}
\begin{split}
\rho_{V}(T)v_{\gamma}&=e^{2i\pi Q(\gamma)}v_{\gamma},\\
\rho_{V}(S)v_{\gamma}&=\frac{{i}^{\frac{b^{-}-b^{+}}{2}}}{\sqrt{|V^{\vee}/V|}}\sum_{\delta\in V^{\vee}/V}e^{-2i\pi (\gamma,\delta)}v_{\delta},
\end{split}
\end{align}
where $\gamma\in V^{\vee}/V$. We denote by $\rho_{V}^{*}$ the dual representation of $\rho_{V}$. 

\begin{remarque}\label{mcgraw}{\normalfont
By a result of McGraw \cite[Prop. 5.6]{mcgraw}, the complex vector space $\mathrm{M}_{k}(\rho_{V}^{*})$ has a rational structure ${\mathrm{M}_{k}(\rho_{V}^{*})}_{\Q}$ given by modular forms with rational coefficients, and similarly for $S_{k}(\rho_{V}^{*})$. }
\end{remarque}
\bigskip

We present an example of a modular form which will be crucial for our later study. 
\begin{example}\label{eisenstein}
{\normalfont
Assume $V$ has signature $(2,b)$ where $b\geq 3$. There is an Eisenstein series $E_{V}$ in ${\mathrm{M}_{k}(\rho_{V}^{*})}$  whose Fourier expansion is given by (see \cite[Prop.4]{bruinierkuss}):  

	$$E_{V}(\tau)=\sum_{\gamma \in V^{\vee}/V}\sum_{n\in -Q(\gamma)+\Z, n\geq 0}c(\gamma,n)q^n v_{\gamma},$$
where $q=e^{2i\pi\tau}$, $\tau\in\mathbb{H}$, and the coefficients $c(\gamma,n)$ are given by: 
\[
\left \{
\begin{array}{r c l}
c(0,0)&=&2\\
c(\gamma,n)&=&-\frac{2^{2+\frac{b}{2}}.\pi^{1+\frac{b}{2}}.n^{\frac{b}{2}}}{\sqrt{|V^{\vee}/V|}\Gamma(1+\frac{b}{2})}.\prod_{p}\mu_p(\gamma,n,V) \quad \textrm{for all}\,n>0,
\end{array}
\right.
\]
where the product is ranging over all primes $p$. The factors $\mu_{p}(\gamma,n,V)$ are defined as follows. For $\gamma\in V^{\vee}/V$, $n\in -Q(\gamma)+\Z$ such that $n$ is positive and $a$ a positive integer, let $$N(\gamma,n,L,a)=|\{\alpha\in L/aL,\, Q(\alpha+\gamma)+n\equiv 0\pmod a\}|.$$

For a prime $p$, Siegel proves in \cite[Hilfssatz 13]{siegeluber} that for $s$ sufficiently large, the value of $p^{-(1+b)s}N(\gamma,n,L,p^s)$ is independent of $s$ and we define 
$$\mu_p(\gamma,n,V):=\lim_{s\rightarrow \infty}p^{-(1+b)s}N(\gamma,n,V,p^s).$$ 
The infinite product $\prod_{p}\mu_p(\gamma,n,V)$ converges as long as  every factor is different from zero. Since $Q$ is indefinite of rank greater than $5$,  this is equivalent  by Hasse-Minkowski theorem (\cite[p.41]{cours}) to the equation $Q(\alpha)+n=0$ having a solution $\alpha$ in $\gamma+V$ . 
In this situation, we say that $n$ satisfy {\it local congruence conditions} and by \cite[Proposition 2]{browning} (see also \cite[Section 11.5]{iwaniec}), we have 
$$\prod_{p}\mu_p(\gamma,n,V)\asymp 1.$$
Hence the estimate 
$$c(\gamma,n)\asymp n^{\frac{b}{2}} .$$
The coefficients $c(\gamma,n)$ are rational numbers by \cite[Proposition 14]{bruinierkuss}
 so that $E_{V}\in {\mathrm{M}_{k}(\rho_{V}^{*})}_{\Q} $.}
\end{example}

\subsection{Borcherds' modular form}\label{general}

In this section we introduce the Heegner divisors and state a modularity result of their generating series. This will allow us later to control the growth of their intersection with a curve $S$ supporting a variation of Hodge structure. As before, let $(V,Q)$ be an even quadratic lattice and assume henceforth that it has signature $(2,b)$ with $b\geq 3$. Let  $\mathrm{O}(V)$ be the orthogonal group of $V$ and $\Gamma_{V}$ the subgroup of elements acting trivially on $V^{\vee}/V$. We refer to \cite{maulik} and \cite[Chapter 6]{huybrechts} for more details on this section. 
 
\subsubsection{The period domain}\label{period}

%The complex vector space $V_{\C}:=L\otimes_{\Z}\C$ is endowed with the scalar extension of $Q$ and the zero locus of $Q$ in the projective space $\P(V_{\C})$ is a smooth quadric.
Let $D_{V}$ be the {\it period domain} associated to $V$, that is the complex analytic variety   $$D_{V}:=\{w\in \P(V_\C),\, (w,w)=0,\, (w,\overline{w})>0 \}.$$ 
%It is a particular instance of Griffiths' period domains which parameterize polarized Hodge structures. In this situation, it parametrizes K3-type Hodge structures on $V$ (see \cite[Chapter 6]{huybrechts}).
%The period domain $D_{V}$ has two connected components which are interchanged by complex conjugation. 
Let $D^{+}_{V}$ be one of the two connected components of $D_{V}$. %The group $\mathrm{O}(V_{\R})$ acts transitively on $D_{V}$, where $V_{\R}:=V\otimes_{\Z}\R$ is endowed with the real extension of $Q$, whereas the connected component of the identity of $\mathrm{O}(V_{\R})$, denoted $G$, acts transitively on  $D^{+}_{V}$. 
Let $G$ be the connected component of the identity of the real Lie group $\mathrm{O}(V_{\R})$, where $V_{\R}:=V\otimes_{\Z}\R$ is endowed with the real extension of $Q$. The action of discrete subgroup $\Gamma^{+}_{V}:=\Gamma_{V}\cap G$ on $D^{+}_{V}$ is proper and totally discontinuous and the quotient $\Gamma^{+}_{V}\backslash D^{+}_{V}$ has the structure of a quasi-projective variety with orbifold singularities by \cite{bailyborel}. 
%, but not free in general. The stabilizer of a point is though always a finite subgroup of $\Gamma_V$. In fact, $D^{+}_{V}$ is a Hermitian symmetric space  
\bigskip 

There is another realization of $D_{V}$ as a Grassmanian. Let $\mathrm{Gr}(2,V_{\R})$ be the Grassmanian of planes of $V_{\R}$ and let $\mathrm{Gr}^{+}(2,V_{\R})$ be the open subset of positive definite planes. We have a natural split covering of degree $2$  
\begin{align*}
D_{V}& \longrightarrow \mathrm{Gr}^{+}(2,V_{\R})\\
\omega =X+iY &\mapsto P=<X,Y>,
\end{align*}
where $P$ is the oriented plane generated by $X$ and $Y$. %The condition on $\omega=X+iY\in D_{V}$ says precisely that $X$ and $Y$ are orthogonal, $Q(X)=Q(Y)>0$ and $P$ is a positive definite plane. 
The restriction of the map above to $D^{+}_{V}$ is a diffeomorphism. Both of the previous descriptions of $D_{V}^{+}$ will be interchangeably used henceforth.  

%By fixing a positive definite plane $P_{0}$ in $V_{\R}$, we get a transitive action:

%\begin{align*}
%G&\rightarrow \mathrm{Gr}^{+}(2,V_{\R})\\
%g	&\mapsto g.P_{0}.
%\end{align*}

%The isotropy subgroup of $P_{0}$  is the maximal compact subgroup $K:=\mathrm{SO}(P_{0})\times\mathrm{SO}(P_{0}^{\bot})$. The induced map $G/K\rightarrow Gr^{+}(2,V_{\R})$ is a diffeomorphism.  
\subsubsection{Heegner divisors}\label{heegnerdivisors}
In this section, the main result from \cite{borcherdszagier} is used to produce linear dependence relations between certain special divisors that  will be defined hereafter.

For any vector $v\in V_{\R}$ such that $Q(v)<0$, let $v^{\bot}$  be the set of planes in  $D^{+}_{V}$ orthogonal to $v$. Let $\gamma\in V^{\vee}/V$ and $n\in Q(\gamma)+\Z$ with $n<0$. The union of hyperplanes 
\begin{align*}
\bigcup_{v\in \gamma+V,\, Q(v)=n}v^{\bot}
\end{align*} is locally finite, invariant under the action of $\Gamma^{+}_{V}$, 
and defines an algebraic divisor  on $\Gamma^{+}_{V}\backslash D^{+}_{V}$ given as 
    	
$$\mathcal{Z}(\gamma,n):=\Gamma^{+}_{V}\backslash\left(\bigcup_{v\in \gamma+V,\, Q(v)=n}v^{\bot}\right).$$
In terms of Hodge structures, $\mathcal{Z}(\gamma,n)$ parametrizes Hodge structures on $V$ for which there exists a rational Hodge class $\lambda$ in $\gamma+V$ with $Q(\lambda)=n$.
\bigskip
 
The restriction of the tautological line bundle $\mathcal{O}(-1)$ to $D^{+}_{V}\subset \P(V_{\C})$ admits a natural $\Gamma^{+}_{V}$-equivariant action and defines an algebraic line bundle  $\mathcal{L}:=\Gamma^{+}_{V}\backslash\mathcal{O}(-1)$ on $\Gamma^{+}_{V}\backslash D^{+}_{V}$ called the {\it Hodge bundle}. We define $\mathcal{Z}(0,0)$ to be a divisor whose class is equal to the  dual of the Hodge bundle. 

Finally, we set $\mathcal{Z}(\gamma,n)=0$ if $n>0$ or if $n=0$ and $\gamma\neq 0$. The $\mathcal{Z}(\gamma,n)$ are the {\it Heegner divisors}. They are Cartier divisors on $\Gamma^{+}_{V} \backslash D^{+}_{V}$, and we denote by $[\mathcal{Z}(\gamma,n)]$ their associated class in $\mathrm{Pic}(\Gamma^{+}_{V} \backslash D^{+}_{V})$. 

Consider the formal power series $$\Phi_{V}(q)=\sum_{\underset{n\in -Q(\gamma)+\mathbb{Z}}{\gamma \in V^{\vee}/V}}[\mathcal{Z}(\gamma,-n)]q^n v_{\gamma} \in \mathrm{Pic}(\Gamma^{+}_{V} \backslash D^{+}_{V})[[q^{\frac{1}{2d}}]]\otimes\mathbb{C}[V^{\vee}/V].$$
Here $d$ is the order of $V^{\vee}/V$.

The following result is due to the work of Borcherds (\cite{borcherdszagier}), combined with the refinement of McGraw (see remark \ref{mcgraw}):
\begin{theorem}
$\Phi_{V}(q)\in \mathrm{Pic}(\Gamma^{+}_{V} \backslash D^{+}_{V})\otimes\mathrm{M}_{1+\frac{b}{2}}(\rho_{V}^{*})_\Q $.
\end{theorem}

We will follow ideas of Maulik in \cite[Section 3]{maulik} with some changes in order to translate the previous theorem in terms of linear dependence relations between the Heegner divisors. This will be achieved by writing $\Phi_V$ as a sum of a multiple of an Eisenstein series and a cusp from, then using standard bounds on the growth of coefficients of cusp forms. 

By \cite[p.27]{bruinier}, for each $\gamma$ in a set of representatives of the quotient of $V^{\vee}/V$ by the involution $x\mapsto -x$,  there exists an Eisenstein series $E_\gamma$ such that the following decomposition holds
\begin{align*}
\mathrm{M}_{1+\frac{b}{2}}(\rho_{V}^{*})_\Q=\bigoplus_{\gamma}\C.E_{\gamma}\oplus \mathrm{S}_{1+\frac{b}{2}}(\rho_{V}^{*})_\Q
\end{align*}
where $E_0=E_V$ is the Eisenstein series from Example \ref{eisenstein}. Since the only $(\gamma,0)$-coefficient of $\Phi_{V}$ which is non-zero is the one corresponding to $\gamma=0$, there exists a finite set $\mathcal{I}$, a family $(\mathcal{Z}(\gamma_i,n_i))_{\in \mathcal{I}}$ of Heegner divisors and a family $(g_i)_{i\in\mathcal{I}}$ of cusp forms such that \begin{align*}\Phi_{V}=\frac{1}{2}[\mathcal{Z}(0,0)]\otimes E_V+\sum_{i\in \mathcal{I}}[\mathcal{Z}(\gamma_i,n_i)]\otimes g_i
\end{align*} 
For $\gamma \in V^{\vee}/V$, $n\in-Q(\gamma)+\Z$, by identifying the $(\gamma,n)$-coefficient in the above expression we get 
\begin{align}\label{picrelation}
\begin{split}
[\mathcal{Z}(\gamma,-n)]=\frac{1}{2}c(\gamma,n)[\mathcal{Z}(0,0)]&+\sum_{i\in\mathcal{I}}g_i(\gamma,n)[\mathcal{Z}(\gamma_i,n_i)].
\end{split}
\end{align}

Notice that all the coefficients in (\ref{picrelation}) are rational numbers.  For a cusp form $f$, the trivial bounds on the order of growth of its coefficients (see \cite[Prop. 1.5.5]{sarnak}) say that 
$$|a_{\gamma,n}(f)|\leq C_{\epsilon,f} n^{\frac{2+b}{4}+\epsilon},$$
for all $\epsilon >0$, some constant $C_{\epsilon,f}>0$, and for all $\gamma \in V^{\vee}/V$ and $n\in -Q(\gamma)+\mathbb{Z}$ with $n\geq 0$. 
 
Hence, we can find a constant $C_{\epsilon}>0$ such that for all $i\in \mathcal{I}$, $n$ and $\gamma$ as before, we have 
$$|g_i(\gamma,n)|\leq C_{\epsilon} n^{\frac{2+b}{4}+\epsilon}.$$  

Taking into account relation (\ref{picrelation}) and the expression in Example \ref{eisenstein}, we get:

\begin{proposition}\label{weak}
For every $\epsilon >0$, $\gamma \in V^{\vee}/V$ and $n\in -Q(\gamma)+\Z$ with $n>0$, the following estimate holds in $\mathrm{Pic}(\Gamma_{V}\backslash D_{V})_{\Q}$
$$[\mathcal{Z}(\gamma,-n)]=-\frac{{(2\pi)^{1+\frac{b}{2}}}n^{\frac{b}{2}}}{\sqrt{|V^{\vee}/V|}\Gamma(1+\frac{b}{2})}\prod_{p}\mu_{p}(\gamma,n,V)[\mathcal{Z}(0,0)]+O_{\epsilon}(n^{\frac{2+b}{4}+\epsilon}).$$
\end{proposition}
The above proposition is a quantitative version of Lemma 3.7 in \cite{maulik}.

\subsection{Extension to a toroidal compactification}

The goal in this section is to extend the estimate in Proposition \ref{weak} to a well chosen toroidal compactification of $\Gamma_{V}\backslash D_{V}$. This will allow us to control their growth in cohomology and the growth of their intersection with any curve as in Theorem \ref{main}. We first start by recalling the construction of the Baily-Borel compactification of $\Gamma_{V}\backslash D_{V}$. For a short summary in the case of orthogonal modular varieties, see \cite{gritsenko} which we follow closely, or \cite[Part III]{borelli} for the general case. 

\subsubsection{Baily-Borel compactification}
There is a  "minimal" compactification of $\Gamma_{V}\backslash D_{V}$ constructed  by Baily and Borel in \cite{bailyborel} and which proceeds by adding  rational boundary components and then showing that the resulting space is a projective algebraic variety.

The rational boundary components correspond precisely to maximal rational parabolic subgroups of $G$, which in turn are the stabilizers of totally isotropic subspaces of $V_{\Q}$. Since $Q$ has signature $(2,b)$, such spaces have dimension $1$ or $2$. Hence, we obtain the following description: 
\begin{align*}
(\Gamma_{V}^{+}\backslash D_{V}^{+})^{BB}=\Gamma_{V}^{+}\backslash D_{V}^{+}\sqcup\bigsqcup_{\Pi}X_\Pi \sqcup\bigsqcup_{\ell} Q_{\ell}.
\end{align*}
where $\ell$ and $\Pi$  run through representatives of the finitely many $\Gamma_{V}^{+}$-orbits of isotropic lines and isotropic planes in $V_{\Q}$. Each $X_{\Pi}$ is a modular curve, and $Q_{\ell}$ is a point. They are also known as {\it $1$-cusps} and {\it $0$-cusps} respectively.

%\subsubsection{Modular forms of orthogonal type}
%Following \cite{borcherds2}, we define the affine cone over $D_{V}^{+}$ to be  $$D_{V}^{\bullet}=\{y\in V_{\C}\backslash \{0\},\, \C y\in D_{V}^{+}\}.$$ 
%\begin{definition}\label{automorphic}
%Let $k\in\Z$ and $\chi:\,\Gamma_{V}^{+}\rightarrow \C^{\times}$ be a character. A  function $F:D_{V}^{\bullet}\rightarrow \C$ is called an automorphic  form of weight $k$ and character $\chi$ for the group $\Gamma_{V}^{+}$ if for all $y\in D_{V}^{\bullet}$, $t\in \C^{\times}$ and $g\in\Gamma_{V}^{+}$ we have
%$$F(ty)=t^{-k}F(y) \quad \textrm{and}\quad F(g.y)=\chi(g)F(y).
%$$
%\end{definition}

\subsubsection{Extension of the relations between Heegner divisors}\label{technical}

The boundary of the Baily-Borel can be singular and the Zariski closure of the Heegner divisors may not be Cartier. To solve this problem,  we extend  the relation (\ref{picrelation}) to  a well-chosen toroidal compactification of $\Gamma_{V}\backslash D_{V}$. We work with the toroidal compactification considered in \cite[Section 5.2]{peterson} and which is given by the perfect cone decomposition. We denote it by $\overline{\Gamma_{V}\backslash D_{V}}^{tor}$. Above each cusp determined by an isotropic subspace $I$ of $V$, the boundary divisors are determined by the one dimensional rays in the $\textrm{Stab}(I)$-invariant decomposition of the positive cone of $I^{\bot}/I$ and in this situation they lie in its boundary. Hence  above every $1$-cusp $F$ there is only one irreducible Cartier boundary divisor $\Delta_F$ and there are no other boundary divisors. Also the closure $\overline{\mathcal{Z}(\gamma,n)}$  of a Heegner divisor $\mathcal{Z}(\gamma,n)$ is Cartier for all $\gamma\in V^{\vee}/V$ and $n\in Q(\gamma)+\Z$. For more details, see \cite[5.2.4]{peterson}. The rest of the section is devoted to bound the coefficients of the boundary divisors in some particular cases.  We start first by recalling Peterson's results in our context, especially Theorem $5.3.3$ in \cite{peterson}.
\bigskip

Let $I$ be an isotropic primitive plane of $V$, $F$ the associated $1$-cusp. The isomorphism class of the definite lattice $I^{\bot}/I$  depends only on the cusp $F$. We denote it by $K_F$ and let $\Theta_{F}$ be the associated theta function, i.e the function defined by 
 $$\Theta_{F}(\tau)=\sum_{\gamma\in K_{F}^{\vee}/K_F}\sum_{x\in\gamma+K_{F}}q^{-Q(x)}v_{\gamma},\,q=e^{2i\pi\tau}, \, \tau\in \mathbb{H},$$ where $(v_{\gamma})_{\gamma\in K_{F}^{\vee}/K_F}$ is the standard basis of $\C[K_{F}^{\vee}/K_F]$. 
   
Let $I^{\#}=I_{\Q}\cap V^{\vee}$. Following \cite[4.1]{brieskorn}, $I$ is said to be {\it strongly primitive} if $I^{\#}=I$. The cardinality $N_{F}$ of the finite group $H_{I}=I^{\#}/I$ depends only on $F$ and is called the imprimitivity of $F$. Let $H_{I}^{\bot}:=\{x\in L^{\vee}/L,\, \forall y\in H_{I},\,(x,y)=0\}$. 
  
\begin{proposition}\label{brieskorn}
Let $I\subset V$ a primitive isotropic plane. Then 
\begin{itemize}
\item[(i)] $H_{I}^{\bot}/H_{I}\simeq K_{F}^{\vee}/K_{F}$ as quadratic finite modules. 
\item[(ii)] $|V^{\vee}/V|=|K_{F}^{\vee}/K_{F}|.N_{F}^{2}$.
\end{itemize}
\end{proposition}
\begin{proof}
Assertion $(i)$ follows from Lemma page $77$ in \cite{brieskorn}. For $(ii)$, notice that $H_{I}^{\bot}\simeq \{\ell\in \mathrm{Hom}(V^{\vee}/V,\Q/\Z),\,\ell_{/H_{I}}=0\}$ and that the cardinality of the latter is equal to $\frac{|V^{\vee}/V|}{N_{F}}$. 
\end{proof}
\bigskip

Let $p:\,H_{I}^{\bot}\rightarrow K_{F}^{\vee}/K_{F}$ be the composite of the projection $H_{I}\rightarrow H_{I}^{\bot}/H_{I}$ followed by the isomorphism $(i)$ from the last proposition. By construction, it is a morphism of quadratic finite modules. 
We have an induced map $p^{*}:\, \C[K_{F}^{\vee}/K_{F}]\rightarrow \C[V^{\vee}/V]$ which maps an element $v_{\gamma}$, $\gamma\in K_{F}^{\vee}/K_{F}$, to  
$$p^{*}v_{\gamma}=\sum_{\underset{p(\delta)=\gamma}{\delta\in H_{I}^{\bot}}}v_{\delta}.$$ 
Using $(ii)$ from the previous proposition, it is straightforward that $p^{*}$ commutes with the action of the metaplectic group $\mathrm{Mp}_{2}(\Z)$ given by the Weil representation as in Section \ref{generalsetting} Equation  (\ref{weil}). Hence, for any $k\in \frac{1}{2}\Z$, we have a map $$p^{*}:\, {\mathrm{M}_{k}(\rho_{K_{F}}^{*})}_{\Q}\rightarrow{\mathrm{M}_{k}(\rho_{V}^{*})}_{\Q}.$$ 
  
For $\gamma \in V^{\vee}/V$, $n\in -Q(\gamma)+\Z$, let $$a(\gamma,n,F)=\frac{N_{F}}{24}(E_{2}.p^{*}(\Theta_{F}))(\gamma,n),$$ where $E_{2}(\tau)=1-24\sum_{n\geq 1}\sigma_{1}(n)q^{n}$, $q=e^{2i\pi\tau}$, $\tau\in \mathbb{H}$, is the weight $2$ Eisenstein series.

The following result is an application of Theorem $5.3.3$ in \cite{peterson} to formula (\ref{picrelation})
\begin{proposition}\label{petersonformula}
Let $\gamma\in V^{\vee}/V$, $n\in -Q(\gamma)+\Z$. Then we have the following linear equivalence relations in $\mathrm{Pic}(\overline{\Gamma_V\backslash D_{V})}^{tor}$
\begin{align}\label{gh}
\begin{split}
[\overline{\mathcal{Z}(\gamma,-n)}]&=\frac{c(\gamma,n)}{2}[\overline{\mathcal{Z}(0,0)}]+\sum_{F\in S_{1}}u(\gamma,n,F)\Delta_F\\&+\sum_{i\in\mathcal{I}}g_{i}(\gamma,n) [\overline{\mathcal{Z}(\gamma_i,n_i)}]+ \sum_{i\in\mathcal{I}}\sum_{F\in S_1}g_{i}(\gamma,n)a(\gamma_i,n_i,F)\Delta_{F},
\end{split}
\end{align}  
where $$u(\gamma,n,F)=\frac{c(\gamma,n)}{2}a(0,0,F)-a(\gamma,n,F),$$ and the coefficients $c(\gamma,n)$ are defined in \ref{eisenstein}.
\end{proposition}
Taking into account the estimates preceding Proposition \ref{weak}, we get 
\begin{proposition}\label{strong}
For every $\epsilon >0$, $\gamma \in V^{\vee}/V$ and $n\in -Q(\gamma)+\Z$ with $n>0$, we have: 
\begin{align*}
 \begin{split}
 [\mathcal{Z}(\gamma,-n)]&=-\frac{{(2\pi)^{1+\frac{b}{2}}}n^{\frac{b}{2}}}{\sqrt{|V^{\vee}/V|}\Gamma(1+\frac{b}{2})}\prod_{p}\mu_{p}(\gamma,n,V)[\mathcal{Z}(0,0)]\\&+\sum_{F\in S_{1}}u(\gamma,n,F)\Delta_{F}+O_{\epsilon}(n^{\frac{2+b}{4}+\epsilon}),
 \end{split}
 \end{align*} in $\mathrm{Pic}(\overline{\Gamma_V\backslash D_V}^{tor})_{\Q}.$
\end{proposition}
\begin{remarque}{\normalfont The term $u(\gamma,n,F)$ can a priori be as large as $c(\gamma,n)$. However, when $F$ is strongly primitive, Lemma \ref{quasiestimate} shows that $c(\gamma,n)$ cancels because of the term $a(\gamma,n,F)$, hence giving a sharper control on the growth of $u(\gamma,n,F)$.}\end{remarque}
 %%%%%%%%%%%%%%%%%%%%%%%%%%%%%%%%%%%%%%%%%%%%%%% Ajouter peut être l'estimée adaptée 
 
\subsection{Some consequences}

We turn now to the consequences of the previous proposition on the distribution of Hodge loci in $1$-dimensional variation of Hodge structure. Let $\{\mathbb{V}_{\Z},\mathcal{F}^{\bullet}\mathcal{V},Q\}$ be a simple, non trivial, polarized variation of Hodge structure over a complex quasi-projective curve $S$ such that the local system $\mathbb{V}_{\Z}^{\vee}/ \mathbb{V}_{\Z}$ is trivial. Let $\rho:\, S\rightarrow\Gamma_{V}\backslash D_{V}$  be the corresponding period map. Let $\overline{S}$ be a smooth compactification of $S$ such that the following diagram is commutative  
\begin{center}
\begin{tikzpicture}[scale=1]
\node (s) at (0,0) {$S$};
\node (s1) at (0,-2) {$\overline{S}$} ;
\node (d) at (2,0) {$ \Gamma_{V}\backslash D_{V}$};
\node (d') at (2,-2) {$\overline{\Gamma_{V}\backslash D_{V}}^{tor}$};
\draw[->,>=latex] (s)--(s1);
\draw[->] (s)--node[above] {$\rho$}(d);
\draw[->,>=latex] (d)--(d');
\draw[->] (s1)--node[above] {$\overline{\rho}$}(d');

\end{tikzpicture}
\end{center}

Let $\gamma\in V^{\vee}/V$ and $n\in -Q(\gamma)+\Z$ such that $n>0$. Since the variation is assumed to be simple, we can express the degree of the divisor $\overline{\rho}^{*}\overline{\mathcal{Z}(\gamma,-n)}$ on $\overline{S}$ as follows:

$$\deg_{\overline{S}}(\overline{\rho}^{*}\overline{\mathcal{Z}(\gamma,-n)})=\sum_{s\in \overline{S}} \mathrm{ord}_{s}(\overline{\rho}^{*}\overline{\mathcal{Z}(\gamma,-n)}),$$ 
where $\mathrm{ord}_{s}(\overline{\rho}^{*}\overline{\mathcal{Z}(\gamma,-n)})$ is the multiplicity  of the intersection of $\overline{S}$ with $\overline{\mathcal{Z}(\gamma,-n)}$  at a point $s\in \overline{S}$.

Notice that $\deg_{\overline{S}}(\overline{\rho}^{*}\overline{\mathcal{Z}(0,0)})=-\mu(S)$, where $\mu$ is the finite measure on $S$ given by integration of the first Chern class of $\mathcal{F}^2\mathcal{V}$. By Proposition \ref{strong} we have:
\begin{corollaire}\label{eqn1}
For every $\epsilon >0$, $\gamma \in V^{\vee}/V$ and $n\in -Q(\gamma)+\Z$ with $n>0$, we have: 
\begin{align*}
\begin{split}
\deg_{\overline{S}}(\overline{\rho}^{*}\overline{\mathcal{Z}(\gamma,-n)})=&\frac{(2\pi)^{1+\frac{b}{2}}n^{\frac{b}{2}}}{\sqrt{|V^{\vee}/V|}\Gamma(1+\frac{b}{2})}\prod_{p}\mu_{p}(\gamma,n,V)\mu(S)\\ &+\sum_{F\in S_1}u(\gamma,n,F)\deg_{\overline{S}}(\overline{\rho}^{*}\Delta_F) +O_{\epsilon}(n^{\frac{2+b}{4}+\epsilon}).
\end{split}
\end{align*}
\end{corollaire}

Assume now that $\deg_{\overline{S}}(\overline{\rho}^{*}\Delta_F)=0$  if $F$ corresponds to a totally isotropic plane which is not strongly primitive. The following lemma gives a control on the coefficient $u(\gamma,n,F)$ when $F$ is associated to a strongly primitive totally isotropic plane. 
\begin{lemma}\label{quasiestimate}
Let $\gamma \in V^{\vee}/V$, $I$ an isotropic, strongly primitive plane of $V$, $F$ the associated $1$-cusp and $K_F=I^{\bot}/I$. Then for all $\epsilon >0 $ we have the following estimate: 
\begin{align*}
|u(\gamma,n,F)|\ll n^{\frac{b}{2}-1+\epsilon}
\end{align*}
\end{lemma}
\begin{proof}
Let $\mathrm{M}^{\leq s}_{k}(\rho_{V}^{*})$ be the vector space of vector-valued quasi-modular from of weight $k$ and depth less than $s$ (see \cite[Definition 1]{imam} and \cite[Section 17.1]{royer} for definitions and properties of quasi-modular forms). Let $D$ be the derivation operator $q\frac{d}{dq}$. Then we have the following structure theorem $$\mathrm{M}^{\leq 1}_{1+\frac{b}{2}}(\rho_{V}^{*})=\mathrm{M}_{1+\frac{b}{2}}(\rho_{V}^{*})\oplus D(\mathrm{M}_{\frac{b}{2}-1}(\rho_{V}^{*})).$$
For a proof, we refer to \cite[Section 17.1]{royer} where it is proven for scalar quasi-modular forms, but the reader may notice that the proof generalizes easily to vector-valued quasi-modular forms.
\bigskip

The product $E_{2}.p^{*}(\Theta_F)$ is an element of $\mathrm{M}^{\leq 1}_{1+\frac{b}{2}}(\rho_{V}^{*})$, hence we can write 
\begin{align}\label{quasi-decomposition}
E_{2}.p^{*}(\Theta_{F})=\sum_{i}\alpha_i E^{i}_{L}+g+D(\tilde{g}),
\end{align}
where $g$ is a cusp form of weight $1+\frac{b}{2}$, $(E_{L}^{i})_{i}$ is a basis of Eisenstein series of $\mathrm{M}_{\frac{b}{2}+1}(\rho_{V}^{*})$ with $E^{0}_{L}=E_{L}$ and $\tilde{g}\in \mathrm{M}_{\frac{b}{2}-1}(\rho_{V}^{*})$. 
By comparing the constant coefficients, we get $\alpha_0=\frac{1}{2}$ and $\alpha_{i}=0$ for $i\neq 0$, since $I$ is strongly primitive. 
Hence for $\gamma\in V^\vee/V$, $n\in -Q(\gamma)+\Z$ with $n\geq 0$, we have 
$$(E_{2}.p^{*} \left(\Theta_F\right))(\gamma,n)= \frac{c(\gamma,n)}{2}+g(\gamma,n)+n\tilde{g}(\gamma,n).$$
Since $\tilde{g}$ is a modular form of weight $\frac{b}{2}-1$, we have $\tilde{g}(\gamma,n)\ll_{\epsilon} n^{\frac{b}{2}-2+\epsilon}$ for all $\epsilon>0$. Also  $g$ is a cusp form and by (see \cite[Prop. 1.5.5]{sarnak}) $|a_{\gamma,n}(f)|\leq C_{\epsilon,f} n^{\frac{2+b}{4}+\epsilon}$. Combining these estimates we get the desired result.
\end{proof}
\begin{remarque}\label{reason1}{\normalfont If $I$ is not strongly primitive, then for $\gamma\notin H_{I}^{\bot}$, we have $u(\gamma,n,F)=c(\gamma,n)$, so the estimate in Lemma \ref{quasiestimate} fails. Even for $\gamma\in H_{I}^{\bot}$, all the Einsenstein series $E_{\delta}$ for $\delta\in H_I$ appear in the decomposition (\ref{quasi-decomposition}) with non-zero coefficients, so again Lemma \ref{quasiestimate} fails. 
}\end{remarque}

In view of the previous lemma, Corollary \ref{eqn1} rewrites  
\begin{corollaire}\label{sharpestimate}
If $\overline{S}$ only meets the boundary of $\overline{\Gamma_{V}\backslash D_{V}}^{tor}$ in divisors $\Delta_{F}$ corresponding to strongly primitive totally isotropic planes, then for every $\epsilon>0$ we have $$\deg_{\overline{S}}(\overline{\rho}^{*}\overline{\mathcal{Z}(\gamma,-n)})=\mu(S)\frac{(2\pi)^{1+\frac{b}{2}}n^{\frac{b}{2}}}{\sqrt{|V^{\vee}/V|}\Gamma(1+\frac{b}{2})}\prod_{p}\mu_{p}(\gamma,n,V)+O_\epsilon(n^{u+\epsilon}),$$
for $\gamma\in V^{\vee}/V$, $n\in -Q(\gamma)+\Z$ with $n>0$ satisfying local congruence conditions of Example \ref{eisenstein} and $u=\max(\frac{b}{2}-1,\frac{2+b}{4})$. If $S$ is projective, then we can choose $u=\frac{2+b}{4}$ 
\end{corollaire} 
\begin{remarque}{\normalfont In the case where the discriminant of $(V,Q)$ is square free, all the primitive isotropic planes are strongly primitive by Proposition \ref{brieskorn}(ii), so the estimate \ref{quasiestimate} holds for all the coefficients $u(\gamma,n,F)$ for $\gamma\in V^{\vee}/V$, $n\in -Q(\gamma)+V$. The condition on the curve $S$ in \ref{sharpestimate} is then automatically satisfied. Notice that here the control on the error term is sharper than the one in Theorem \ref{main}. This is because we don't know how to bound the intersection of $\overline{S}$ and $\overline{\mathcal{Z}(\gamma,-n)}$ at the boundary points, see Remark \ref{reason1}. However we conjecture that $|S\cap \mathcal{Z}(\gamma,-n)|_{mult}$ grows as the main term in the corollary.   
}
\end{remarque}

\section{Equidistribution in orthogonal modular varieties}
%Let $S$ be a quasi-projective subvariety of $\Gamma^{+}_{V} \backslash D^{+}_{V}$. The restriction of the form $\omega$ is again a Kähler form on $S$. If $S$ is a curve, then integration with respect  to this form defines a measure $\mu$ on $S$ of finite volume.  %We first recall some classical facts and constructions on Hermitian symmetric spaces. For more details, we refer to \cite{helgason}. Following standard conventions, we always denote by $\mathfrak{a}$ the complexified Lie algebra associated to a real Lie algebra $\mathfrak{a}_{0}$. 
The main goal of this section is to prove Proposition \ref{local} which gives a lower estimate on the growth of the Hodge locus. The results in this section are independent from those in section 2.
\subsection{Construction of a local map}\label{local1}

%A variation of  Hodge structure over $U$ is the datum of a local system $\mathbb{V}_{\Z}$ of free $\Z$-modules of finite rank, a holomorphic decreasing filtration $\mathcal{F}^{i}\mathcal{V}$, $0\leq i\leq 2,$ over the  vector bundle $\mathcal{V}=\mathbb{V}_{\Z}\otimes_{\C}\mathcal{O}_{U}$ endowed with an integrable flat connection $\nabla:\mathcal{V}\rightarrow \mathcal{V}\otimes \Omega_{U}^{1}$  whose sheaf of locally constant sections is $\mathbb{V}_{\C}:=\mathbb{V}_{\Z}\otimes\C$ and which satisfies  the Griffiths' transversality (see \cite[Chap.17]{voisin}):	
%$$\nabla(\mathcal{F}^{\ell}\mathcal{V})\subset \mathcal{F}^{\ell-1}\mathcal{V}\otimes \Omega_{S}^{1}. $$
%For each $u\in U$, we require that  the filtration $\mathcal{F}^{i}\mathcal{V}_{u},0\leq i\leq 2$, induces a pure Hodge structure of weight $2$ over $\mathbb{V}_{\Z,u}$. The variation is said to be polarized if there exists a flat, non-degenerate symmetric bilinear form $Q$ on $\mathbb{V}_{\Z}$ which polarizes the Hodge structure on $\mathbb{V}_{\Z,u}$, i.e for every $s\in S$: 
%\begin{enumerate}
%\item $Q(\mathcal{V}^{p,q}_{u},\mathcal{V}^{p',q'}_{u})=0$, unless $p=q'$ and $q=p'$.
%\item$i^{p-q}Q(v,\overline{v})<0$ for every $v\in \mathcal{V}^{p,q}_{u}\backslash \{0\}$. 
%\end{enumerate}

Let $U$ be a connected complex manifold and let $\{\mathbb{V}_{\Z},\mathcal{F}^{\bullet}\mathcal{V},Q\}$ be an integral, polarized variation of Hodge  structure  of weight $2$ over $U$ with $h^{2,0}=1$. Assume that the fiber of $(\mathbb{V}_{\Z},Q)$ at a point $u_0$ (hence at all points of $U$) is isomorphic to a quadratic even lattice $(V,Q)$ of signature $(2,b)$ as in Section \ref{general} and assume also that the local system $\mathbb{V}_{\Z}^{\vee}/\mathbb{V}_{\Z}$ is trivial. It follows that the  monodromy representation factors through $\Gamma_V$, the stable orthogonal group of $(V,Q)$.  Let $\rho:\, U\rightarrow \Gamma_L\backslash D_{V}$ be the corresponding period map. We will construct in  this section a sphere bundle over $U$ that keep track of Hodge classes and a map from the latter to the quadric $A_{0}=\{x\in V_\R,\, Q(x)=-1\}$.  
%$\widetilde{U}\xrightarrow{p} U$ be the universal cover of $U$.% The pullback $p^{-1}\mathbb{V}_{\Z}$ is a trivial local system on $U$ isomorphic to $\underline{V}_{\widetilde{U}}$.
%We have thus a period map 
%\begin{center}
%\begin{tikzpicture}[scale=1]
%\node (d) at (2,2) {$D_{V}^{+}$};
%\node (u) at (-2,0) {$U$} ;
%\node (u') at (-2,2) {$\widetilde{U}$};
%\node (d') at (2,0) {$\Gamma_L\backslash D_{V}$};
%\draw[->,>=latex] (d')--(d);
%\draw[->,>=latex] (u')--node[above]{$\widetilde{\rho} $}(d);
%\draw[->,>=latex] (u')--(u);
%\draw[->,>=latex] (u)--node[above]{$\rho $}(d');
%\end{tikzpicture}
%\end{center} 
\bigskip

The line bundle $\mathcal{F}^{2}\mathcal{V}$ is simply the pullback of the Hodge bundle $\mathcal{L}$ via $\rho$. Let $\mathcal{V}_{\R}$ be the real vector bundle whose sheaf of differentiable sections is equal to $\mathbb{V}_{\Z}\otimes_{\Q}\mathcal{C}^{\infty}_{\R}$, where $\mathcal{C}^{\infty}_{\R}$ is the sheaf of $\mathcal{C}^{\infty}$ real-valued functions on $U$. The fiber at a point $u\in U$ of $\mathcal{V}_{\R}$ is isomorphic to $V_{\R}$. This vector bundle contains a sub-vector bundle that we shall note $\mathcal{V}^{1,1}_{\R}$ and whose sheaf of differentiable sections is 
$$\mathcal{F}^{1}\mathcal{V}\otimes \mathcal{C}^{\infty}_{\C}\cap \mathbb{V}_{\Z}\otimes\mathcal{C}^{\infty}_{\R}.$$
Let $\mathcal{V}^{1,1}:=\mathcal{F}^{1}\mathcal{V}/\mathcal{F}^{2}\mathcal{V}$. Then $\mathcal{V}^{1,1}_{\R}$ is the real part of $\mathcal{V}^{1,1}$, i.e the fiber at each point $u$ is equal to $\mathcal{V}_{u}^{1,1}\cap V_{\R}$. 

Assume that $U$ is simply connected. Parallel transport by the Gauss-Manin connection trivializes  the vector bundle $\mathcal{V}_{\R}$, hence it is isomorphic to  $U\times V_{\R}$ and this isomorphism preserves the intersection form. Thus one has the commutative diagram

\begin{center}
\begin{tikzpicture}[scale=1]
\node (k) at (-2,1) {$\mathcal{V}^{1,1}_{\R}$};
\node (D) at (0,-1) {$U$} ;
\node (l) at (2,1) {$U\times V_{\R}$};
\draw[->,>=latex] (k)--(D);
\draw[right hook->,>=latex] (k)--(l);
\draw[->,>=latex] (l)--(D);
\end{tikzpicture}
\end{center}
\bigskip

Projecting forward to $V_{\R}$, we get the parallel transport  map: 
$$\Xi:\, \mathcal{V}^{1,1}_{\R}\rightarrow V_{\R}.$$ 
The locus where this map is not submersive were studied in \cite[17.3.4]{voisin} and goes back to Griffiths and Green. Let us recall the setting and the main result. By Griffiths' transversality, the integrable connection 
$$\nabla:\mathcal{V}\rightarrow\mathcal{V}\otimes \Omega^{1}_{U}$$ induces a $\mathcal{O}_{U}$-linear map:  
$$\overline{\nabla}:\mathcal{F}^{1}\mathcal{V}/\mathcal{F}^{2}\mathcal{V}\rightarrow \mathcal{F}^{0}\mathcal{V}/\mathcal{F}^{1}\mathcal{V}\otimes \Omega^{1}_{U}$$
Let $u\in U$, then taking the fibers at $u$ induce a $\C$-linear map 
$$ \overline{\nabla}_u:\mathcal{V}^{1,1}_u\rightarrow  \mathcal{V}^{0,2}_u\otimes \Omega^{1}_{U,u}$$ 
Then we have the following lemma, due to Green (see Lemma  17.21 from \cite{voisin}). %For the convenience of the reader, we recall the proof: 
\begin{lemma}\label{submersion0}
Let $u\in U$, $\lambda \in \mathcal{V}_{u}^{1,1}$. If the map $$\overline{\nabla}_u(\lambda):T_uU\rightarrow \mathcal{V}^{0,2}_u$$ is surjective then $\Xi$ is submersive at $(u,\lambda)$. 
\end{lemma} 

Consider the fibration over $U$ defined by  $$\mathcal{S}_{U}=\{(u,\lambda),\, u\in U,\, \lambda \in \mathcal{V}^{1,1}_{u,\R},\, Q(\lambda)=-1 \}\rightarrow U.$$ 
For every $u\in U$, the restriction of the quadratic form $Q$ to $\mathcal{V}^{1,1}_{u,\R}$ is negative definite, and  the fiber $\mathcal{S}_{U,u}$ is thus a $(b-1)$-dimensional sphere. 
By restriction of $\Xi$, we get a map:
$$\phi:\,\mathcal{S}_{U}\rightarrow A_{0},$$
where $A_{0}=\{x\in V_\R,\, Q(x)=-1\}$.
\begin{lemma}\label{submersion}
Let $u\in U$, $\lambda \in \mathcal{V}_{u}^{1,1}$ such that $Q(\lambda)=-1$. If the map $$\overline{\nabla}_u(\lambda):T_uU\rightarrow \mathcal{V}^{0,2}_u$$ is surjective then $\phi$ is submersive at $(u,\lambda)$.
\end{lemma}

\begin{proof}
Let $u$ and $\lambda$ be as in the statement. The following diagram is commutative   
\begin{center}
\begin{tikzpicture}[scale=1]
\node (s) at (-4,0) {$T_{(u,\lambda)}\mathcal{S}_{U}$};
\node (v) at (0,0) {$T_{(u,\lambda)}\mathcal{V}^{1,1}_{\R}$} ;
\node (r) at (4,0) {$\R$};
\node (a) at (-4,-2) {$T_{\lambda}A_0$};
\node (v') at (0,-2) {$T_{\lambda}V_\R$};
\node (r') at (4,-2) {$\R$}; 
\draw[->,>=latex] (s)--(v);
\draw[->,>=latex] (v)--node[above]{$d_{(u,\lambda)}(Q\circ\Xi)$}(r);
\draw[->,>=latex] (s)--node[right]{$d_{(u,\lambda)}\phi$}(a);
\draw[->,>=latex] (a)--(v');
\draw[->,>=latex] (v')--node[above]{$d_{\lambda}Q $}(r');
\draw[->,>=latex] (r)--(r');
\draw[->,>=latex] (v)--node[right]{$d_{(u,\lambda)}\Xi$}(v');
\end{tikzpicture}
\end{center}

The rows  are exact by construction. By Lemma \ref{submersion0}, $d_{(u,\lambda)}\Xi$  is surjective . Hence the map $d_{(u,\lambda)}\phi$  is surjective which proves the lemma. 
\end{proof}
%The same proof of lemma \ref{sumbersion} shows  that  

If $\rho(U)$ is not a point, then for  $u\in U$ outside the locus where the differential of $\rho$ is identically zero, there exists $\lambda\in\mathcal{S}_{U,u}$ which satisfies the  condition of \ref{submersion}. Hence, the image $\mathrm{Im}(\phi)$ is open around $\lambda$. In particular,  the set of points of $U$ for which $\mathcal{V}^{1,1}_u$ contain an extra rational Hodge class $x$ with $Q(x)=-1$ is dense (see \cite[Proposition 17.20]{voisin} and \cite[Theorem 1.1]{oguiso} for a proof without the norm condition imposed by $Q$).
\bigskip

Lemma \ref{submersion} shows that on order to study the distribution of the Hodge locus in $U$, one can first study the distribution of points $\lambda$ in $A_{0}$ for which there exists $\gamma\in V^{\vee}/V$ and $n\in -Q(\gamma)+\N$ such that $\sqrt{n}\lambda\in \gamma+V$, since the locus where $\phi$ is not submersive is a proper real analytic subset of $\mathcal{S}_U$. Hence it is negligible from a measure-theoretic perspective. This will be explained in the following section.  
\subsection{Eskin-Oh's equidistribution result}
The study of the distribution of Hodge locus in $U$ amounts via the map $\phi$ constructed above to the study of radial projections of integral points of $V_{\R}$ on $A_{0}$. We will need thus to understand  the  distribution in $A_{0}$  of the set $\{\lambda\in A_0, \, \sqrt{n}\lambda \in \gamma+V,\}$ for $\gamma\in V^{\vee}/V$ and $n\in -Q(\gamma)+V$ with $n>0$. This a is a well studied problem and can be dealt with using Hardy-Littlewood's circle method (see \cite{vaughan}). The results we present here follow  \cite{eskinoh} and \cite{oh} to which we refer for more details. Recall that $G=\mathrm{O}(V_{\R})^{+}$ is the connected component of the identity of the real Lie group  $\mathrm{O}(V_{\R})$.
\bigskip 

Let $\mu_{\infty}$ be the  $G$-invariant measure on $A_0$  defined in the following way : take $W$ an open subset of $V_{\R}$ and let 
$$\mu_{\infty}(W\cap A_0)=\lim_{\epsilon \rightarrow 0}\frac{\mathrm{Leb}\left(\{x\in W,\, |Q(x)+1|<\epsilon\}\right)}{2\epsilon}.$$
Here $\mathrm{Leb}$ is the Lebesgue measure on $V_{\R}$ for which the lattice $V$ is of covolume $1$.  
We can now state the main result of this section which is an application of Theorem 1.2 in \cite{eskinoh} (see also \cite[Section 5]{oh}) and the Siegel mass formula \cite[(1.6)]{eskinoh}. 
\begin{proposition}\label{eskin}
Let $\Omega$ be a compact subset of $A_0$ with zero measure boundary, $\gamma\in V^{\vee}/V$ and $n\in -Q(\gamma)+\Z$ with $n>0$. Then 
$$|\{\lambda\in \gamma+V,\, \frac{1}{\sqrt{n}}\lambda\in \Omega\}|\sim\mu_{\infty}(\Omega).n^{\frac{b}{2}}. \prod_{p}\mu_{p}(\gamma,n,V),$$
as $n\rightarrow +\infty$.
\end{proposition}
\begin{proof}
To see how Theorem 1.2 from \cite{eskinoh} can be applied to our situation, we refer to the proof of Theorem 6.1 in {\it loc. cit.}. The only difference is that here we don't restrict to fundamental discriminants so we need to check that condition (1.3) in \cite{eskinoh} holds. In other words, we need to know that for each $n_{0}$, there is only finitely many $n$ such that 
\begin{align} \label{stable}
\frac{1}{\sqrt{n}}(\gamma+V)\cap A_0=\frac{1}{\sqrt{n_0}}(\gamma+V)\cap A_0.
\end{align} 
\bigskip

For a given $n$, notice that if (\ref{stable}) holds, then  $$\mathcal{Z}(\gamma,-n)=\mathcal{Z}(\gamma,-n_0),$$
so by Corollary \ref{eqn1} this can be true only for finitely many $n$. 
\end{proof}
\bigskip

The $G$-invariant measure  $\mu_{\infty}$ can be recovered as integration of a $G$-invariant volume form on $A_{0}$. Indeed, the group $G$ acts transitively on $A_0$ and the choice of an element $\xi$ in $A_0$ determines a surjective map 
\begin{align}\label{action1}
\begin{split}
\pi_{\xi}:\,G&\longrightarrow A_0 \\
g&\mapsto  g.\xi.
\end{split}
\end{align}
Let $H$ be the stabilizer of $\xi$. The induced map  $G/H\rightarrow A_0$ is a diffeomorphism giving $A_0$ the structure of a symmetric space. 
Let $\mathfrak{g}_{0}$ and $\mathfrak{h}_{0}$ be the Lie algebras of $G$ and $H$ respectively. Then $\mathfrak{g}_{0}/\mathfrak{h}_{0}$ is isomorphic to the tangent space of $A_0$ at $\xi$ via the differential of $\pi_{\xi}$ at the identity of $G$. The space of $G$-invariant volume forms on $A_{0}$ is then identified with $\bigwedge^{b+1}(\mathfrak{g}_{0}/\mathfrak{h}_{0})^{\vee}$.  
\bigskip

Let $(e_{1},e_{2},\xi_{1},\dots,\xi_{b})$ be an orthogonal basis of $V_{\R}$  such that for $i=1,2$ and $j=1,\dots,b$ we have $Q(e_{i})=-Q(\xi_j)=1$. Let $\omega_{A_0}$ be the unique $G$-invariant volume form on $A_0$ such that   
\begin{align}\label{relation1}
\omega_{A_0,\xi_{1}}=de_1\wedge de_2\wedge d\xi_{2}\wedge\dots\wedge d\xi_{b}
\end{align}
in $\bigwedge^{b+1}(T_{\xi_1}A_0)^{\vee}$. Let $\mu_{A_{0}}$ be the $G$-invariant measure on $A_{0}$ given by integration of $\omega_{A_{0}}$. 
%seen as an element in $\left(\bigwedge^{b+1}(\widetilde{\mathfrak{p}}\oplus \mathfrak{s}^{b-1})\right)^{\vee}$, and where $(Y^{i}_{1},\dots,Y^{i}_{b-1})$ is an orthogonal basis of $\mathfrak{s}^{b-1}$.
We have then the following proposition: 
\begin{proposition}\label{secondmesure}
For every open subset $W$ of $A_0$, we have 
$$\mu_{\infty}(W)=\frac{2^{\frac{b}{2}}}{\sqrt{|V^{\vee}/V|}}\mu_{A_0}(W)$$
\end{proposition}
\begin{proof}
%Let $\xi\in A_0$ and let $(e_{1},e_{2},\xi_{1},\dots,\xi_{b})$ be a basis of $V_\R$ as above such that that $\xi=\xi_1$. 
It is enough to prove the equality for $W$ open subset of $A_{0}$ containing $\xi_{1}$. There exists an open subset $U^{b+1}$ in $\R^{b+1}$ containing $0$ such that the map  
\begin{align*}
U^{b+1}&\rightarrow A_{0}\\
(x_{1},x_{2},y_{2},\dots,y_{b})&\mapsto (x_{1},x_{2},\sqrt{x_{1}^2+x_{2}^2-y_{2}^{2}-\dots-y_{b}^{2}+1},y_{2},\dots,y_{b}).
\end{align*}
is a local chart around $\xi_{1}$. Let $W$ be its image.  For $\epsilon >0$, the image of the map 

\begin{align*}
U^{b+1}\times ]-\epsilon,\epsilon[&\rightarrow A_{0}\\
(x_{1},x_{2},y_{2},\dots,y_{b},r)&\mapsto (x_{1},x_{2},\sqrt{x_{1}^2+x_{2}^2-y_{2}^{2}-\dots-y_{b}^{2}+1+r},y_{2},\dots,y_{b}).
\end{align*}
defines a tubular neighbourhood $W_\epsilon$ of $W$ in $\R^{b+2}$ and one can check  that  
$$\lim_{\epsilon\rightarrow 0}\frac{1}{2\epsilon}\int_{W_\epsilon}\omega=\frac{1}{2}\mu_{A_{0}}(W), $$
where $\omega=de_{1}\wedge de_{2}\wedge d\xi_{1}\wedge \dots\wedge d\xi_{b}$ and 
$A_{\epsilon}=\{x\in V_\R,\, |Q(x)+1|<\epsilon\}.$ 
By change of variable, we have 
\begin{align*}
\mathrm{Leb}(\{ x\in W_\epsilon, |Q(x)+1|<\epsilon\})= \frac{2^{1+\frac{b}{2}}}{\sqrt{|V^{\vee}/V|}}\int_{W_\epsilon}\omega.
\end{align*}

Hence 
\begin{align*}
\mu_{\infty}(W)&=\lim_{\epsilon\rightarrow 0}\frac{\mathrm{Leb}(\{ x\in W, |Q(x)+1|<\epsilon\})}{2\epsilon}\\
&=\frac{2^{\frac{b}{2}}}{\sqrt{|V^{\vee}/V|}}\mu_{A_{0}}(W)
\end{align*}
which proves the lemma.
\end{proof}

\begin{corollaire}\label{eskincor}
Let $\Omega$ be a compact subset of $A_0$ with zero measure boundary, $\gamma\in V^{\vee}/V$ and $n\in -Q(\gamma)+\Z$ with $n>0$. Then
$$|\{\lambda\in \gamma+V,\, \frac{1}{\sqrt{n}}\lambda\in \Omega\}|\sim \mu_{A_0}(\Omega).\frac{2^{\frac{b}{2}}}{\sqrt{|V^{\vee}/V|}}.n^{\frac{b}{2}}. \prod_{p}\mu_{p}(\gamma,n,V),$$
as $n\rightarrow +\infty$.
\end{corollaire}

%This can be translated into the following relation
%\begin{align}\label{relation1}
%\pi_{\xi_{i}}^{*}\omega_{A}=dU_{i}\wedge dV_{i}\wedge \pi_{\xi_{i}}^{*}\omega_{\mathbb{S}^{b-1}}.
%\end{align}
%Let $\omega_{2}$ be an invariant volume form on $\mathrm{SO}(P_0)$. The forms $\omega_{G}=\omega_{2}\wedge\omega_{n}\wedge\omega^{b-1} $, $\omega_{H}=\omega_{2}\wedge\omega_{n-1}\wedge\omega^{b-1}$ define invariant forms on $G$ and $H$ respectively. Moreover, the form $\omega_{b-1}\wedge\omega$ define a $G$-invariant volume form on $A$ such that the associated measure $\mu_{G/H}$ satisfies the following equality, for every measurable function $f$ with compact support on $G$ 

%$$\int_{G}f(g)d\mu_{G}(g)=\int_{G/H}\left(\int_{H}f(gh)d\mu_{H}(h)\right)d\mu_{g}.$$

%where $d\mu_{G}$ and $d\mu_{H}$ are the measures on $G$ and $H$ given by integration of $\omega_{G}$ and $\omega_{H}$ respectively. 

\subsection{Quantitative study of the Hodge locus}\label{quantitative}
The goal of this section is to put together results from the previous sections in order to prove Proposition \ref{local} which gives a lower bound on the cardinality of the Hodge locus. Let $\{\mathbb{V}_{\Z},\mathcal{F}^{\bullet}\mathcal{V},Q\}$ be a non-trivial, polarized, simple variation of Hodge structure over a complex quasi-projective curve $S$ and let $\rho: S\rightarrow \Gamma_{V}^{+}\backslash D^{+}_{V}$ be the associated period map. 
\bigskip
 
Recall that the Chern class $\omega$ of the Hodge bundle $\mathcal{F}^{2}\mathcal{V}$ defines a volume form on $S$ . For any open subset $\Delta\subset S$, we note $\mu(\Delta)=\int_{\Delta}\omega$. Let $\Delta$ be an open simply connected subset of $S$. The restriction of  $\rho$  to $\Delta$ lifts to $D_{V}^{+}$. Let $0\in\Delta$ be a point in $\Delta$ and $P_{0}$ the positive definite plane associated to $\rho(0)$. Then $P_{0}$ defines a maximal compact subroup $K:=\mathrm{SO}(P_0)\times\mathrm{SO}(P_0^{\bot})$ of  $G$ and a diffeomorphism 
\begin{align*}
\pi:\,G/K&\rightarrow D^{+}_{V}\\
g&\mapsto g.P_{0}
\end{align*}  
We constructed in the previous paragraph a map $$\phi:\,\mathcal{S}_{\Delta}\rightarrow A_{0}$$ where $\mathcal{S}_{\Delta}$ is a sphere bundle over $\Delta$ that fits into the following commutative diagram
\begin{center}
\begin{tikzpicture}[scale=1]
\node (k) at (-10,1) {$\mathcal{S}_{\Delta}$};
\node (D) at (-8,-1) {$\Delta$} ;
\node (l) at (-6,1) {$\Delta\times A_0$};
\node (a) at (-4,1) {$A_0$};
\draw[->,>=latex] (k)--(D);
\draw[right hook->,>=latex] (k)--(l);
\draw[->,>=latex] (l)--(D);
\draw[->,>=latex] (l)--(a);
\draw[->,>=latex] (k) to [bend left=20] node[above]{$\phi$}  (a);
\end{tikzpicture}
\end{center}

For any $U\subset S$, $\gamma\in V^{\vee}/V$ and $n\in -Q(\gamma)+\Z$ with $n>0$, let $$|U\cap \mathcal{Z}(\gamma,-n)|_{mult}=\sum_{s\in U\cap \mathcal{Z}(\gamma,-n)}m(s,\gamma,n),$$
where $m(s,\gamma,n)=|\{\lambda\in \mathcal{S}_{\Delta,s},\,\sqrt{n}\lambda\in \gamma+V\}|.$
\begin{lemma}\label{int}
Let $s\in S$, $\gamma\in V^{\vee}/V$ and $n\in -Q(\gamma)+\Z$ such that $n>0$. Then  
$$m(s,\gamma,n)\leq \mathrm{ord}_{s}(\rho^{*}\mathcal{Z}(\gamma,n)) .$$
\end{lemma}
\begin{proof}

Let $\gamma$, $n$ and $s$ as in the statement of the proposition. Assume that 
$$\{\lambda\in \mathcal{S}_{\Delta,s},\,\sqrt{n}\lambda\in \gamma+V\}=\{\lambda_{1},\dots,\lambda_{k}\},$$
where $k=m(s,\gamma,n)$. There exists a finite index congruence subgroup $\Gamma$ of $\Gamma_{V}^{+}$ such that the orbits $\Gamma.\lambda_{1},\dots, \Gamma.\lambda_{k}$ are pairwise disjoints. In particular, the divisor $$\mathcal{Z'}:= \Gamma\backslash\left(\bigcup_{\underset{Q(\lambda)=-n}{\lambda\in \gamma+V}}\lambda^{\bot}\right)\subset \Gamma\backslash D_{V}^{+}$$ has at least $k$ irreducible components. The kernel of the morphism $\pi_{1}(S)\rightarrow \Gamma_{V}/\Gamma$ defines a finite étale cover $S'\xrightarrow{\iota'} S$  and we have a commutative diagram 
\begin{center}
\begin{tikzpicture}[scale=1]
\node (s) at (0,0) {$S'$};
\node (s1) at (0,-2) {$S$} ;
\node (d) at (2,0) {$ \Gamma\backslash D_{V}^{+}$};
\node (d') at (2,-2) {$\Gamma_{V}^{+}\backslash D_{V}^{+}$};

\draw[->,>=latex] (s)--node[left]{$\iota'$}(s1);
\draw[->] (s)--node[above] {$\rho'$}(d);
\draw[->,>=latex] (d)--node[right]{$\iota$}(d');
\draw[->] (s1)--node[above] {$\rho$}(d');

\end{tikzpicture}
\end{center}

Remark that $\iota^{*}\mathcal{Z}(\gamma,n)$ is  equal to $\mathcal{Z}'$. Let $s'\in S'$ such that $\iota'(s')=s$. Then $\mathrm{ord}_{s'}(\rho'^{*}\iota^{*}\mathcal{Z}(\gamma,n))\geq k$, since $\mathcal{Z}'$ has at least $k$ irreducible components. By commutativity of the diagram above,  
$$\rho'^{*}\iota^{*}\mathcal{Z}(\gamma,n)=\iota'^{*}\rho^{*}\mathcal{Z}(\gamma,n).$$ 
Since $\iota'$ is étale, we have $$\mathrm{ord}_{s'}(\iota'^{*}\rho^{*}\mathcal{Z}(\gamma,n))=\mathrm{ord}_{s}(\rho^{*}\mathcal{Z}(\gamma,n)),$$
 which yields the desired result.   
\end{proof}
\begin{remarque}{\normalfont
We only have an inequality here because the curve $S'$ may have intersection multiplicity strictly greater than one with a given irreducible component of the Heegner divisor $Z'$. In fact Theorem \ref{main} implies that this does not happen when $n$ is sufficiently large.  
}
\end{remarque}
\begin{proposition}\label{local}
Let $\gamma\in V^{\vee}/V$. For all $s\in S$, there exists a simply connected open neighborhood  $\Delta\subset S$ of $s$ such that $$\liminf_{n}\frac{|\Delta\cap \mathcal{Z}(\gamma,-n)|_{mult}}{n^{\frac{b}{2}}\prod_{p<\infty}\mu_{p}(\gamma,n,V)}\geq\frac{{(2\pi)^{1+\frac{b}{2}}}}{\sqrt{|V^{\vee}/V|}\Gamma(1+\frac{b}{2})}\mu(\Delta),$$
where $n>0$ ranges over numbers in $-Q(\gamma)+\Z$ represented by $-Q$ in $\gamma+V$.
\end{proposition}
\begin{proof}

The map $\pi:\,G\rightarrow D_{V}^{+}\simeq G/K$ is submersive. Hence there exists $U$ a simply connected open subset of $D_{V}^{+}$ around $P_{0}$ such that the following diagram is commutative: 
\begin{center}
\begin{tikzpicture}[scale=1]
\node (u) at (0,0) {$U$};
\node (uk) at (2,2) {$U\times K$} ;
\node (pu) at (-2,2) {$\pi^{-1}(U)$};
\draw[->,>=latex] (uk)--(u);
\draw[->] (pu)--node[above] {$\sim$}(uk);
\draw[->,>=latex] (pu)--(u);

\end{tikzpicture}
\end{center}

Assume that $\rho(\Delta)$ is contained in $U$. We have a holomorphic map $\Upsilon:\,\Delta\rightarrow G$ given by the composition 
\begin{align*}
\Delta&\xrightarrow{\rho} U\rightarrow U\times K \xrightarrow{\sim} \pi^{-1}(U)\\
t&\mapsto \rho(t)\mapsto (\rho(t),1_K)\mapsto \Upsilon_t\\
\end{align*}
 
Hence, we have a local trivialization of the fibration $\mathcal{S}_{\Delta}$ given by 

\begin{center}
\begin{tikzpicture}[scale=1]
\node (delta) at (0,0) {$\Delta$};
\node (a) at (2,2) {$\mathcal{S}_{\Delta}$} ;
\node (deltas) at (-2,2) {$\Delta\times \mathbb{S}^{b-1}$};
\node (depart) at (-2,1.5) {$(t,\lambda)$};
\node (arrivee) at (2,1.5) {$(t,\Upsilon_t.\lambda)$};
\draw[->,>=latex] (depart)--(u);
\draw[->] (deltas)--node[above] {$\sim$}(a);
\draw[->,>=latex] (arrivee)--(u);
\draw[->] (depart)--(arrivee);

\end{tikzpicture}
\end{center}

where  $\mathbb{S}^{b-1}=\{x\in P_{0}^{\bot},\, Q(x)=-1\}$.
\bigskip

The map $$\phi:\, \mathcal{S}_{\Delta}\rightarrow A_0$$ is, by lemma \ref{submersion},  submersive at $(t,\lambda)$ if the map $$\overline{\nabla}_u(\lambda):T_uU\rightarrow \mathcal{V}^{0,2}_u$$ is surjective, or equivalently not-identically zero, since $T_{u}U$ is of dimension $1$ over $\C$. 
Let $\mathcal{S}_{\Delta}^{sing}$ the locus where $\phi$ is not submersive. Then $\mathcal{S}_{\Delta}^{sing}$ is a proper real analytic closed subset of $\mathcal{S}_{\Delta}$ negligible for the Lebesgue measure. Outside $\mathcal{S}_{\Delta}^{sing}$, $\phi$ is submersive and in fact a local diffeomorphism by equality of dimensions.

Let $\psi$ be the composite map $$\psi:\,\Delta\times \mathbb{S}^{b-1}\rightarrow \mathcal{S}_{\Delta}\xrightarrow{\phi}A_{0}.$$ The pullback $\phi^{*}\omega_{A_{0}}$ is a volume form on $\mathcal{S}_{\Delta}$ and so is $\psi^{*}\omega_{0}$ on $\Delta\times \mathbb{S}^{b-1}$.

Let $\epsilon >0$, and let $\mathcal{S}_{\Delta}^{sing,\epsilon}$ be an open subset containing  $\mathcal{S}_{\Delta}^{sing}$ such that $\int_{\mathcal{S}_{\Delta}^{sing,\epsilon}}\phi^{*}\omega_{A_{0}}\leq\epsilon$. Up to shrinking $\Delta$, we  can find a finite open cover $W_i$ of $\mathcal{S}_{\Delta}\backslash \mathcal{S}_{\Delta}^{sing,\epsilon}$ such the restriction $\phi_i$ of $\phi$ to $W_i$ is a diffeomorphism. 

By Corollary \ref{eskincor}, we get 
%\sum_{s\in \Delta} |\{\lambda\in \mathcal{S}_{\Delta,s},\,\sqrt{n}\lambda\in \gamma+V\}|\\&\geq \sum_{i}|\{(s,\lambda)\in W_{i},\,\sqrt{n} \lambda \in \gamma+V \}|\\&=&= \sum_{i}\int_{\mathrm{Im}(\phi_{i})}\omega_{A_{0}}.\frac{2^{\frac{b}{2}}.n^{\frac{b}{2}}}{\sqrt{|V^{\vee}/V|}}. \prod_{p}\mu_{p}(\gamma,n,V)+o(n^{\frac{b}{2}})\\
\begin{align*}
\begin{split}
|\Delta\cap \mathcal{Z}(\gamma,-n)|_{mult} &\geq\sum_{i}|\{\lambda\in \mathrm{Im}(\phi_{i}),\, \sqrt{n}\lambda\in\gamma +V\}|\\
&=\sum_{i}\int_{W_{i}}\phi^{*}\omega_{A_{0}}.\frac{2^{\frac{b}{2}}.n^{\frac{b}{2}}}{\sqrt{|V^{\vee}/V|}}. \prod_{p}\mu_{p}(\gamma,n,V)+o(n^{\frac{b}{2}})\\
&\geq\frac{{2}^{\frac{b}{2}}.\left(\int_{\mathcal{S}_{\Delta}}\phi^{*}\omega_{A_0}-\epsilon\right).n^{\frac{b}{2}}}{\sqrt{|V^{\vee}/V|}} . \prod_{p}\mu_{p}(\gamma,n,V)+o(n^{\frac{b}{2}})\\
\end{split}
\end{align*} 
Here we used that $\int _{\mathcal{S}_{\Delta}}\phi^{*}\omega_{A_{0}}=\mu(\Delta).\frac{2.\pi^{1+\frac{b}{2}}}{\Gamma(1+\frac{b}{2})},$ a result we prove in Lemma \ref{crucial} below. Hence we have  

\begin{align*}
\liminf_{n}\frac{|\Delta\cap \mathcal{Z}(\gamma,-n)|_{mult}}{n^{\frac{b}{2}}\prod_{p<\infty}\mu_{p}(\gamma,n,V)}\geq  \frac{{(2\pi)^{1+\frac{b}{2}}}}{\sqrt{|V^{\vee}/V|}\Gamma(1+\frac{b}{2})}\mu(\Delta)-\frac{2^{\frac{b}{2}}}{\sqrt{|V^{\vee}/V|}}.\epsilon
\end{align*}
By letting $\epsilon \rightarrow 0$, we get the desired result.  
\end{proof}
\bigskip
\begin{lemma}\label{crucial}
We have:
$$\int _{\mathcal{S}_{\Delta}}\phi^{*}\omega_{A_{0}}=\mu(\Delta).\frac{2.\pi^{1+\frac{b}{2}}}{\Gamma(1+\frac{b}{2})}.$$
\end{lemma}
\begin{proof}
The differential of the map $\pi:\,G\rightarrow D_{L}^{+}$  at the identity of $G$ induces an isomorphism of $\mathfrak{p}_{0}$ with the tangent space of $D_{L}^{+}$ at $P_{0}$. Since $\omega$ is a $G$-invariant $2$-form, it corresponds uniquely to an element ${\Omega}$ of $\bigwedge^{2} \mathfrak{p}_{0}^{\vee}$. Fix an orthogonal basis $(e_{1},e_{2},\xi_{1},\dots,\xi_{b})$ of $V_{\R}$ compatible with the decomposition $V_{\R}=P_{0}\oplus P_{0}^{\bot}$ and such that $Q(e_{i})=-Q(\xi_{j})=1$ for $i=1,2$ and $j=1,\dots,b$. The Lie algebra $\mathfrak{g}_{0}$ is then identified with  $\mathfrak{so}(2,b)$ and  an element $M\in\mathfrak{so}(2,b)$ is written by blocks in the following way 
\begin{align}\label{matrix}
\begin{pmatrix}
0&\theta&U\\
-\theta&0&V\\
^tU& ^tV&N\\
\end{pmatrix}
\end{align}
where $\theta\in \R$, $U$ and $V$ are $1\times b$-dimensional real matrices, and  $N$ is a $b\times b$-dimensional antisymmetric real matrix. %The sub-Lie algebra $\mathfrak{k}_{0}$ of $K$ can be identified with the space of elements in $\mathfrak{so}(2,b)$ for which $U=V=0$ in the description above, and $\mathfrak{p}_{0}$ is identified with the space of  matrices of the form :
%$$\begin{pmatrix}0&0&U\\0&0&V\\^tU& ^tV&0\\\end{pmatrix}.$$
For $i,j=1,\dots,b+2$, let $E_{i,j}$ be the matrix whose coefficients are zero except the  coefficient $(i,j)$ which is equal to $1$. For $i=1,\dots,b$, define $U_{i}=E_{1,2+i}+E_{2+i,1}$ and $V_{i}=E_{2,2+i}+E_{2+i,2}$. The family $(U_{i},V_{i})_{i=1,\dots,b}$ is  a basis of $\mathfrak{p}_{0}$. 

By \cite[13.1]{carlson} (see also \cite[5.3]{goresky}), the curvature $\Theta$  of the Hodge bundle is given by $$\Theta(X,Y)=-\lambda([X,Y])$$ for $X,Y \in \mathfrak{p}$ and where $\lambda$ is the linear form on $\mathfrak{so}(2,b)$ associating to an element $M$ written as in (\ref{matrix}) the element $i\theta \in \C$. Notice that $\lambda$ is the differential of the generator $\chi$ of the group of characters of $K$ whose associated automorphic line bundle is the Hodge bundle (see \cite{zucker}). A  computation shows that
\begin{align}\label{hodge}
\Omega=\frac{i}{2\pi}\Theta=\frac{1}{2\pi}\sum_{i=1}^{b}dU_{i}\wedge dV_{i}.
\end{align}

Recall that the Killing form $B$ of $\mathfrak{g}_{0}$ is negative definite on the Lie algebra $\mathfrak{k}_{0}$ of $K$ and we have thus an orthogonal decomposition (see \cite{helgason}):
$$\mathfrak{g}_{0}=\mathfrak{k}_{0}\oplus\mathfrak{p}_{0}.$$ 
Let $\xi\in P_{0}^{\bot}$ such that $Q(\xi)=-1$. Then $K'=\mathrm{SO}(P_{0})\times\mathrm{SO}((\R\xi\oplus P_{0})^{\bot})$
 is a maximal compact subgroup of $H$ and we have similarly an orthogonal decomposition: 
  $$\mathfrak{h}_{0}=\mathfrak{k'}_{0}\oplus\mathfrak{p'}_{0}.$$
Let $\mathfrak{s}^{b-1}$ and $\widetilde{\mathfrak{p}}$ the orthogonal complements of $\mathfrak{k'}_{0}$ and $\mathfrak{p'}_{0}$ in $\mathfrak{k}_{0}$ and $\mathfrak{p}_{0}$ respectively with respect to the Killing form $B$ of $\mathfrak{g}_{0}$: 
\begin{align*}
\mathfrak{k}_{0}=\mathfrak{k'}_{0}\oplus \mathfrak{s}^{b-1}\quad,\quad \mathfrak{p}_0=\mathfrak{p'}_0\oplus \widetilde{\mathfrak{p}}.
\end{align*}  
The  quotient $\mathfrak{g}_{0}/\mathfrak{h}_{0}$ can then be identified to $\mathfrak{s}^{b-1}\oplus \widetilde{\mathfrak{p}}.$
The space $\mathfrak{s}^{b-1}$  can be identified, via the differential at the identity of $\mathrm{SO}(P_{0}^{\bot})$ of the map $\pi_{\xi_{1}}$ introduced in (\ref{action1}), with the tangent space at $\xi_{1}$ of the sphere $$\mathbb{S}^{b-1}:=\{x\in P_{0}^{\bot},\, Q(x)=-1\},$$ which explains the notation. 
Let $\omega_{\mathbb{S}^{b-1}}$ be the unique $\mathrm{SO}(P_{0}^{\bot})$-invariant volume form on $\mathbb{S}^{b-1}$ such that $$\omega_{\mathbb{S}^{b-1},\xi_{1}}=d\xi_{2}\wedge\dots\wedge d\xi_{b}.$$ Then $(\mathbb{S}^{b-1},\omega_{\mathbb{S}^{b-1}})$  is isometric to a sphere of dimension $b-1$ and radius $1$ with its standard volume form, hence 
\begin{align}\label{computation}	
\int_{\mathbb{S}^{b-1}}\omega_{\mathbb{S}^{b-1}}=\frac{b.{\pi}^{\frac{b}{2}}}{\Gamma(1+\frac{b}{2})}.
\end{align}
\bigskip

The group $G$ acts transitively on the left on $\mathcal{S}:=\mathcal{S}_{D_{L}^{+}}$ via $g.(w,\xi)=(g.w,g.\xi)$ for $g\in G$, $w\in D_{V}^{+}$ and $\xi \in \mathcal{S}_{u}$. The map $\phi:\,\mathcal{S}\rightarrow A_{0}$ is  $G$-equivariant. For each $i=1,\dots,b$, we have thus a surjective map 
\begin{align*}
p_{i}:\, G&\rightarrow \mathcal{S}\\
g&\mapsto (g.P_{0},g.\xi_{i}) 
\end{align*}
that fits into the following commutative diagram
\begin{center}
\begin{tikzpicture}[scale=1]
\node (g) at (0,0) {$G$};
\node (a') at (-2,-2) {$\mathcal{S}$} ;
\node (a) at (2,-2) {$A_0$};
\draw[->,>=latex] (g)--node[right]{$\pi_{\xi_{i}}$}(a);
\draw[->,>=latex] (a')--node[below] {$\phi$}(a);
\draw[->,>=latex] (g)--node[left]{$p_i$}(a');
\end{tikzpicture}
\end{center}
The differential of $p_i$ induces an isomorphism between the tangent space of $\mathcal{S}$ at $(P_0,\xi_i)$ and $\mathfrak{s}^{b-1}\oplus \mathfrak{p}_{0}$, where $\mathfrak{s}^{b-1}$  is isomorphic to the tangent space of $\mathbb{S}^{b-1}$ at $\xi_{i}$. The element $dU_{i}\wedge dV_{i} \in \bigwedge^{2}\mathfrak{p}_{0}^{\vee}$ defines a $G$-invariant $2$-form on $\mathcal{S}$ that we denote by  $\omega_{i}$.  
Let \begin{align*}
t_{i}:\,K&\rightarrow \mathbb{S}^{b-1}\\
k&\rightarrow k.\xi_{i}.
\end{align*} 
The pull back of the form $\omega_{\mathbb{S}^{b-1}}$ along $t_{i}$ is identified to an element $dY^{i}_{1}\wedge\dots dY^{i}_{b-1}$ of $\bigwedge^{b-1}\mathfrak{k}_{0}^{\vee}$ for an orthogonal family $(Y^{i}_{1},\dots,Y^{i}_{b-1})$ of $\mathfrak{k}_{0}$. Let $\omega^{(i)}$ be the $G$-invariant $(b-1)$-from on $\mathcal{S}$ such that $p_{i}^{*}\omega^{(i)}$ is equal to  $dY^{i}_{1}\wedge\dots dY^{i}_{b-1}$ in $\bigwedge^{b-1}\mathfrak{g}_{0}^{\vee}$

For each $i=1,\dots,r$, we have by (\ref{relation1}) $$p_{i}^{*}\phi^{*}\omega_{A_0}=\pi_{\xi_{i}}^{*}\omega_{A_0}=dU_i\wedge dV_i\wedge dY^{i}_{1}\wedge\dots dY^{i}_{b-1}.$$
which is equal to $p_{i}^{*}\omega_{i}\wedge p_{i}^{*}\omega^{(i)}=p_{i}^{*}(\omega_{i}\wedge\omega^{(i)})$ by the construction itself. Hence, $\phi^{*}\omega_{A_0}=\omega_{i}\wedge\omega^{(i)}$ and summing over $i$ yields \begin{align*}
\phi^{*}\omega_{A_{0}}=\frac{1}{b}\sum_{i=1}^{b}\omega_{i}\wedge \omega^{(i)}
\end{align*}

Notice now that the restrictions to $\mathbb{S}^{b-1}$ of the forms $\omega^{(i)}$  are all equal to the form $\omega_{\mathbb{S}_{b-1}}$ and that $\sum_{i=1}^{b}\omega_{i}=2\pi \omega$ by (\ref{hodge}). Hence 
$$\psi^{*}\omega_{A_{0}}=\frac{2\pi}{b}\omega\wedge\omega_{\mathbb{S}^{b-1}}.$$
We have in particular
$$\frac{2\pi}{b}\mu(\Delta).\int_{\mathbb{S}^{b-1}}\omega_{\mathbb{S}^{b-1}}=\int _{\mathcal{S}_{\Delta}}\phi^{*}\omega_{A_{0}}.$$ 
Combined with (\ref{computation}), this yields the desired result.
\end{proof}

\section{End of the proof and applications}
The goal of the section is to prove Theorem \ref{main}. We keep the notations from previous sections, i.e  $\{\mathbb{V}_{\Z},\mathcal{F}^{\bullet}\mathcal{V},Q\}$ is a simple, non trivial, polarized variation of Hodge structure over a quasi-projective curve $S$ such that the local system $\mathbb{V}_{\Z}^{\vee}/ \mathbb{V}_{\Z}$ is trivial and  $\rho:\, S\rightarrow\Gamma_{V}\backslash D_{V}$  is the corresponding period map.
\subsection{First reduction}
Let $H$ be a maximal isotropic subgroup of $V^{\vee}/V$ with respect to $Q$ and let $\overline{\mathbb{V}}_{\Z}$ be the inverse image in $\mathbb{V}_{\Z}$ of $\underline{H}_{S}$, the trivial local system of fiber $H$. Then $\{\overline{\mathbb{V}}_{\Z},\mathcal{F}^{\bullet}\mathcal{V},Q\}$ defines a simple, non-trivial, polarized variation of Hodge structure over $S$. Moreover, the fibers of the local system $\overline{\mathbb{V}}_{\Z}$ are isomorphic to a lattice $\overline{V}$ which has only strongly primitive totally isotropic planes and $\overline{V}^{\vee}/\overline{V}\simeq H^{\bot}/H$. 
\begin{proposition}\label{reduction}
If Theorem \ref{main} holds for $\{\overline{\mathbb{V}}_{\Z},\mathcal{F}^{\bullet}\mathcal{V},Q\}$ then it holds for $\{\mathbb{V}_{\Z},\mathcal{F}^{\bullet}\mathcal{V},Q\}$. 
\end{proposition} 
\begin{proof}
Let $\Delta$ be an open simply connected subset of $S$ which satisfies Proposition \ref{local} and let $\gamma\in H^{\bot}$ and $n\in -Q(\gamma)+\Z$. Denote by $\overline{\gamma}$ its image in $H^{\bot}/H\simeq \overline{V}^{\vee}/\overline{V}$. Then 
%\begin{align*}
%\{s\in \Delta,\, \exists \lambda \in \overline{\gamma}+\mathbb{V'}_{\Z,s},\, Q(\lambda)=-n\}=
%\bigsqcup_{t\in H}\Delta\cap \mathcal{Z}(\gamma+t,-n).
%\end{align*}
\begin{align*}
|\Delta\cap \mathcal{Z}_{\overline{V}}(\overline{\gamma},-n)|_{mult}=\sum_{t\in H}|\Delta\cap \mathcal{Z}(\gamma+t,-n)|_{mult}
\end{align*}
where $\mathcal{Z}_{\overline{V}}(\overline{\gamma},-n)$ is the Heegner divisor associated to the lattice $(\overline{V},Q)$, to $\overline{\gamma}$ and $n$. By assumption, Theorem \ref{main} holds for $\{\overline{\mathbb{V}}_{\Z},\mathcal{F}^{\bullet}\mathcal{V},Q\}$ : 
\begin{align*}|\Delta\cap \mathcal{Z}_{\overline{V}}(\overline{\gamma},-n)|_{mult}=-\mu(\Delta)\frac{\overline{c}(\overline{\gamma},n)}{2}+o(n^{\frac{b}{2}}).
\end{align*}
The $\overline{c}(\overline{\gamma},n)$ are the coefficients of the Eisenstein series $E_{\overline{V}}$ constructed out of the lattice $(\overline{V},Q)$ in a similar fashion to Example \ref{eisenstein}. For $t\in H$, we have by Lemma \ref{local} 
$$|\Delta\cap \mathcal{Z}(\gamma+t,-n)|_{mult}\geq -\mu(\Delta)\frac{c(\gamma+t,n)}{2}+o(n^{\frac{b}{2}})$$
Thus
$$|\Delta\cap \mathcal{Z}(\gamma,-n)|_{mult}\leq \frac{\mu(\Delta)}{2}.\left(-\overline{c}(\overline{\gamma},n)+\sum_{t\in H\backslash \{0\}}c(\gamma+t,n)\right)+o(n^{\frac{b}{2}})$$
Using Lemma \ref{curiousrelation} below, we have 
$$|\Delta\cap \mathcal{Z}(\gamma,-n)|_{mult}\leq -\mu(S)\frac{c(\gamma,n)}{2}+o(n^{\frac{b}{2}})$$
Combined with \ref{local}, we get the desired result. 
\end{proof}
\begin{lemma}\label{curiousrelation}
Let $\gamma\in H^{\bot}$, $n\in -Q(\gamma)+\Z$. Then 
$$\sum_{t\in H}c(\gamma+t,n)=\overline{c}(\overline{\gamma},n)+O_\epsilon(n^{\frac{b+2}{4}+\epsilon})$$ 
\end{lemma}
\begin{proof}
Let $p:H^{\bot}\rightarrow H^{\bot}/H\simeq \overline{V}^{\vee}/\overline{V}$. Then $p$ induces two morphisms 
\begin{align*}
p_{*}:\, \C[V^{\vee}/V]&\rightarrow  \C[\overline{V}^\vee/\overline{V}]\\
			v_{\gamma}&\mapsto v_{p(\gamma)}\, \textrm{if}\, \gamma \in H^{\bot},\, 0 \, \textrm{otherwise}.						
\end{align*}
and
\begin{align*}
 p^{*}:\, \C[\overline{V}^\vee/\overline{V}]&\rightarrow \C[V^\vee/V]\\
			v_{\delta}&\mapsto \sum_{\gamma\in H^{\bot},\,p(\gamma)=\delta}v_{\gamma}			 
\end{align*}
 which commutes with the Weil representation on both sides. %$\C[\overline{V}^\vee/\overline{V}]$ and $\C[V^\vee/V]$. 
 Hence we have two $\C$-linear map: $p_*:\mathrm{M}_{1+\frac{b}{2}}(\rho_{V}^{*}) \rightarrow\mathrm{M}_{1+\frac{b}{2}}(\rho_{\overline{V}}^{*}) $ and $p^*:\mathrm{M}_{1+\frac{b}{2}}(\rho_{\overline{V}}^{*})\rightarrow \mathrm{M}_{1+\frac{b}{2}}(\rho_{V}^{*})$.
The modular form $p^*E_{\overline{V}}-p^*p_*E_{V}$ is then a cuspidal form and Lemma \ref{curiousrelation} follows by identifying its coefficients. 
\end{proof}

\subsection{An upper bound}
By Theorem \ref{reduction}, we may assume that $V^{\vee}/V$ has no non-trivial totally isotropic subgroup in order the prove Theorem \ref{main}. Hence all the primitive isotropic planes of $V$ are strongly primitive (see the definition preceding Proposition \ref{brieskorn}). Let $\overline{S}$ be a smooth compactification of $S$ which fits in the following commutative diagram   
\begin{center}
\begin{tikzpicture}[scale=1]
\node (s) at (0,0) {$S$};
\node (s1) at (0,-2) {$\overline{S}$} ;
\node (d) at (2,0) {$ \Gamma_{V}\backslash D_{V}$};
\node (d') at (2,-2) {$\overline{\Gamma_{V}\backslash D_{V}}^{tor}$};

\draw[->,>=latex] (s)--(s1);
\draw[->] (s)--node[above] {$\rho$}(d);
\draw[->,>=latex] (d)--(d');
\draw[->] (s1)--node[above] {$\overline{\rho}$}(d');

\end{tikzpicture}
\end{center}
\bigskip 

The boundary $\overline{S}\backslash S$ is  finite. Let $\Delta_{0}$ be a finite union of open subsets of $\overline{S}$  around each of those points. Consider $\Delta$ an  open subset in $S\backslash \Delta_{0}$ which satisfies lemma \ref{local}. We can find a finite disjoint family of open  subsets $(\Delta_{i})_{i\in I}$ included in $S$ which satisfy lemma \ref{local} and such that $\mu(S)=\mu(\Delta)+\mu(\Delta_{0})+\sum_{i\in I}\mu(\Delta_{i})$. For each $i\in I$, we have:
$$\liminf_{n} \frac{|\Delta_i\cap \mathcal{Z}(\gamma,-n)|_{mult}}{n^{\frac{b}{2}}\prod_{p<\infty}\mu_{p}(\gamma,n,V)}\geq\frac{{(2\pi)^{1+\frac{b}{2}}}}{\sqrt{|V^{\vee}/V|}\Gamma(1+\frac{b}{2})} \mu(\Delta_{i}),$$
for $\gamma\in V^{\vee}/V$ and $n\in -Q(\gamma)+\Z$ satisfying local congruence conditions. 
Also, by lemma \ref{int}
\begin{align*}
|\Delta\cap \mathcal{Z}(\gamma,-n)|_{mult} \leq \deg_{\overline{S}}(\overline{\rho}^{*}\overline{\mathcal{Z}(\gamma,-n)})-\sum_{i\in I}|\Delta_{i}\cap \mathcal{Z}(\gamma,-n)|_{mult}
\end{align*}
Hence by Corollary \ref{sharpestimate}
\begin{align*}
\limsup_n \frac{|\Delta\cap \mathcal{Z}(\gamma,-n)|_{mult}}{n^{\frac{b}{2}}\prod_{p<\infty}\mu_{p}(\gamma,n,V)}&\leq\frac{{(2\pi)^{1+\frac{b}{2}}}\left(\mu(S)-\sum_{i}\mu(\Delta_{i})\right)}{\sqrt{|V^{\vee}/V|}\Gamma(1+\frac{b}{2})} \\
&\leq\frac{{(2\pi)^{1+\frac{b}{2}}}\left(\mu(\Delta)+\mu(\Delta_{0})\right)}{\sqrt{|V^{\vee}/V|}\Gamma(1+\frac{b}{2})}.
\end{align*}
Since the volume of $\Delta_{0}$ can be chosen arbitrarily small, we deduce that 
$$ \limsup_n \frac{|\Delta\cap \mathcal{Z}(\gamma,-n)|_{mult}}{n^{\frac{b}{2}}\prod_{p<\infty}\mu_{p}(\gamma,n,V)}\leq \frac{{(2\pi)^{1+\frac{b}{2}}}}{\sqrt{|V^{\vee}/V|}\Gamma(1+\frac{b}{2})}\mu(\Delta).$$
Combined with Lemma \ref{local}, this yields the desired equidistribution result.  

\subsection{Elliptic fibrations in families of K3 surfaces}
We now derive some equidistribution results in quasi-polarized families of K3 surfaces. We begin by some background on K3 surfaces. The main references are \cite{huybrechts} and \cite{barth}.

\bigskip

%\subsection{Background on K3 surfaces}
%\subsection{Period domain of polarized \texorpdfstring{K3}{K3} surfaces}
%\begin{definition}
%A complex K3 surface is a compact simply connected complex manifold of dimension $2$ and with trivial canonical bundle.
%\end{definition}
Let $X$ be a K3 surface. The second cohomology group with integer coefficients of $X$ endowed with  its intersection form $(.,.)$ is an even unimodular lattice of signature $(3,19)$, hence isomorphic abstractly to the {\it K3 lattice}  
$$\Lambda_{K3}=U^{\oplus 3}\oplus E_{8}(-1)^{\oplus 2},$$
where $U$ is the hyperbolic lattice and $E_{8}$ the unique definite positive even unimodular lattice of rank $8$, up to isomorphism. Denote by $Q=\frac{(\cdot.\cdot)}{2}$ the associated quadratic form. 

\begin{definition}
An elliptic K3 surface is a projective K3 surface $X$ together with a surjective morphism $\pi\,:X\rightarrow \mathbb{P}^{1}$ such the generic fiber is a smooth integral curve of genus one. 
\end{definition}

Recall that there  is an elliptic fibration on $X$ if, and only if there exists a {\it parabolic line bundle} on $X$, i.e a non-trivial line bundle $L$ with $(L.L)=0$. Indeed, if $X$ admits an elliptic fibration $\pi:\,X\rightarrow \P^1$, then the class of a fiber gives a non-trivial element $e\in \mathrm{Pic}(X)$ such that $(e,e)=0$. 
Conversely, let $L$ be a non-trivial line bundle  with square zero. Either $L$ or $L^{-1}$ is effective by Riemann-Roch. Assume $L$ is effective. In \cite[{\bf 8}.2.13]{huybrechts}, it is shown that up to acting on $L$ by the Weyl group of $X$, we can assume that $L$ is nef of square zero. Then [{\bf 2}.3.10, {\it loc.cit.}] shows that $\pi_{V}:X\rightarrow \P(\mathrm{H}^{0}(X,L)^{\vee})$ factors through $\P^1$ and induces an elliptic fibration whose fiber class is equal to $L$.

%\begin{proposition}
%Let $X$ be a complex projective K3 surface. 
%\begin{itemize}
%\item[(i)]
%\item[(ii)] If $\rho(X)\geq 5$, then $X$ admits an elliptic fibration.
%\end{itemize} 
%\end{proposition}
%The $(\mathrm{ii})$ follows from $(i)$ and from Hasse-Minkowski theorem asserting that any indefinite quadratic form of rank at least $5$ represents zero. 
Let $P\subset \Lambda_{K3}$ be a primitive Lorentzian anistropic sublattice of rank $\rho\leq 4$ and let $V=P^{\bot}$. Then $(V,Q)$ is an even quadratic lattice of signature $(2,20-\rho)$ and we have an isomorphism of quadratic finite modules $(V^{\vee}/V,Q)\simeq(P^{\vee}/P,-Q)$ (see\cite[Prop.{\bf 14}.0.2]{huybrechts}). Recall that a $P$-K3 surface is a  K3 surface $X$  with a fixed primitive embedding $P\rightarrow \mathrm{Pic}(X)$ such that the image of $P$ contains a quasi-polarization $\ell$. 
If $L\in\mathrm{Pic}(X)$ is of square $0$, we can write $L=L_P+L_V$ where $L_{P}\in P^{\vee}$ and $L_V\in V^{\vee}$. Then $(L.L)=(L_P.L_P)+(L_V.L_V)$  and $(L_V.L_V)\leq 0$ since the restriction of the form $(\,,\,)$ to $V$ is negative definite. Hence $(L_P.L_P)>0$,  unless $L=L_P$, which is excluded since $P$ is assumed to be anisotropic.  
\begin{definition}\label{degree}
Let $X$ be a $P$-K3 surface, $\gamma\in P^{\vee}/P$ and $n\in Q(\gamma)+\Z$. A parabolic line bundle $L$ on $X$ is said to be of type $(\gamma,n)$ if $L_P\in \gamma+P$ and $(L_P.L_P)=2n$. An elliptic fibration is said to be of type $(\gamma,n)$  if a line bundle defining the fibration is so. We call $n$ the  norm of the elliptic fibration. 
\end{definition}
We are now in the setting of Section \ref{general} and we follow its notations, namely $D_{V}$ is  the period domain associated to the lattice $(V,Q)$ and $\mathcal{Z}(\gamma,n)$ is the Heegner divisors associated to $\gamma\in V^{\vee}/V$ and $n\in Q(\gamma)+\Z$.
\begin{proposition}\label{intersection}
Let $X$ be a $P$-K3 surface, $\gamma\in P^{\vee}/P$ and $n\in Q(\gamma)+\Z$. Then $X$ admits a parabolic line bundle of type $(\gamma,n)$ if and only if there exists $t\in \gamma+P$ such that $(t.t)=2n$ and the period of $X$ lies on the Heegner divisor $\mathcal{Z}(\gamma,-n)$.
\end{proposition}
\begin{proof}
Let $L$ be a line bundle on $X$ defining an elliptic fibration of type $(\gamma,n)$. Write $L=L_P+L_V$ as above. Then the element $L_V\in \gamma+V$ satisfies $(L_V.L_V)=-2n$ and take $t=L_P$. Hence, the period of $X$ lies on the Heegner divisor $\mathcal{Z}(\gamma,-n)$.
Conversely, if the period of $X$ lies in $\mathcal{Z}(\gamma,-n)$, then there exists $\lambda\in H^{1,1}(X)\cap (\gamma+V)$ such that $(\lambda,\lambda)=-2n$. By assumption, there exists $t\in \gamma+V$ such that $(t.t)=2n$. Then  $L=\lambda+t\in \mathrm{Pic}(X)$ is of square zero and non-trivial.
\end{proof}
\begin{proposition}\label{intersection3}
Let $X$ be a $P$-K3 surface. Then $X$ admits an elliptic fibration of  norm less than $n$ if and only if the period of $X$ lies on the union of the Heegner divisors $\mathcal{Z}(\gamma,-s)$ for $\gamma\in P^{\vee}/P$ and $s\in ]0,n]$ represented by $Q$ in $\gamma+P$.
\end{proposition}
\begin{proof}
The forward direction is clear. For the converse, we can construct a parabolic line bundle $L$ on $X$ of norm less than $n$ in the same way as it was done above. If $L$ is nef, then $L$ defines an elliptic fibration of degree less than $n$ and we are done. Otherwise, there exists a $-2$-curve $C$ such that $(L.C)<0$. Then $s_C(L):=L+(L,C).C$ is a parabolic line bundle with positive intersection with $C$ and  of norm less than the norm of $L$. We repeat the process if $s_{C}(L)$ is not nef. After a finite number of actions by the Weil group, we get a nef line bundle. 
\end{proof}

\begin{proof}[Proof of corollary \ref{k3}]
Let  $\mathcal{X}\xrightarrow{\pi}S$ be a non-isotrivial family of K3 surfaces  with generic Picard group equal to $P$. The orthogonal to $\underline{P}_{\mathcal{X}}$ in $R^{2}\pi_{*}\underline{\Z}_{\mathcal{X}} $ defines a polarized variation of Hodge structure of weight $2$ over $S$ with fibers isomorphic to the lattice $(V,Q)$ and to which we can apply Theorem \ref{main}. Using Proposition \ref{intersection}, this proves $(i)$ and $(ii)$. For $(iii)$, let $\Delta \subset S$ an open subset and $\widetilde{N}(n,\Delta)$ the number of $s\in \Delta$ (counted with multiplicity) for which $\mathcal{X}_{s}$ admits an elliptic fibration of norm less than $n$. Then by Proposition \ref{intersection3} and Theorem \ref{main} we have
 \begin{align*}
 \frac{\widetilde{N}(n,\Delta)}{\widetilde{N}(n,S)}=\frac{\sum_{\gamma\in V^{\vee}/V}\sum_{s\leq n, s\in Q(\gamma+P)} |\Delta\cap Z(\gamma,-s)|_{mult}}{\sum_{\gamma\in V^{\vee}/V}\sum_{s\leq n, s\in Q(\gamma+P)} |S\cap Z(\gamma,-s)|_{mult}}
 \underset{n\rightarrow \infty}{\longrightarrow} \frac{\mu(\Delta)}{\mu(S)}
 \end{align*}
\end{proof}

\begin{remarque}{\normalfont There is an analogous result which concerns families of hyperkähler manifolds and which we state below. See \cite{hk} for definitions. Indeed, given a hyperkähler manifold, the {\it Beauville-Bogomolov-Fujiki} form, defined in \cite{beauville1}, endows its second integral Betti cohomology group with a structure of a lattice of signature $(3,b_{2}-3)$.}\end{remarque}
\begin{corollaire}\label{hyperkähler}
Let $d$ be an integer. Let $(\mathcal{X},\mathcal{L}_{2d})\rightarrow S$ be a non-isotrivial, split quasi-polarized family of hyperkähler manifolds of degree $2d$ over a quasi-projective curve $S$ with generic Picard rank equal to $1$  and let $\{R^{2}\pi_{*}\underline{\Z}_{\mathcal{X}},\mathcal{F}^{\bullet}\mathcal{H}\}$ be the induced variation of Hodge structure over $S$. Let $\mu$ be the measure induced by integrating the first Chern class of $\mathcal{F}^{2}\mathcal{H}$. Then the set of points $s\in S$ for which $\mathcal{X}_{s}$ admits a  parabolic line bundle $L$ such that $(L.\mathcal{L}_{2d,s})=2dn$  becomes equidistributed in $S$ with respect to $\mu$ as $n\rightarrow +\infty$.
\end{corollaire}
For the definition of split polarization, we refer to \cite[Definition 3.9]{gritsenko1}.

\bibliographystyle{alpha}
\bibliography{bibliographie}

\begin{thebibliography}{BHPVdV04}

\bibitem[BB66]{bailyborel}
W.~L. Baily, Jr. and A.~Borel.
\newblock Compactification of arithmetic quotients of bounded symmetric
  domains.
\newblock {\em Ann. of Math. (2)}, 84:442--528, 1966.

\bibitem[BD08]{browning}
T.~D. Browning and R.~Dietmann.
\newblock On the representation of integers by quadratic forms.
\newblock {\em Proc. Lond. Math. Soc. (3)}, 96(2):389--416, 2008.

\bibitem[Bea83]{beauville1}
Arnaud Beauville.
\newblock Vari\'et\'es {K}\"ahleriennes dont la premi\`ere classe de {C}hern
  est nulle.
\newblock {\em J. Differential Geom.}, 18(4):755--782, 1983.

\bibitem[BHPVdV04]{barth}
Wolf~P. Barth, Klaus Hulek, Chris A.~M. Peters, and Antonius Van~de Ven.
\newblock {\em Compact complex surfaces}, volume~4 of {\em Ergebnisse der
  Mathematik und ihrer Grenzgebiete. 3. Folge. A Series of Modern Surveys in
  Mathematics [Results in Mathematics and Related Areas. 3rd Series. A Series
  of Modern Surveys in Mathematics]}.
\newblock Springer-Verlag, Berlin, second edition, 2004.

\bibitem[BJ06]{borelli}
Armand Borel and Lizhen Ji.
\newblock {\em Compactifications of symmetric and locally symmetric spaces}.
\newblock Mathematics: Theory \& Applications. Birkh\"auser Boston, Inc.,
  Boston, MA, 2006.

\bibitem[BK01]{bruinierkuss}
Jan~Hendrik Bruinier and Michael Kuss.
\newblock Eisenstein series attached to lattices and modular forms on
  orthogonal groups.
\newblock {\em Manuscripta Math.}, 106(4):443--459, 2001.

\bibitem[BKPSB98]{bkpsb}
Richard~E. Borcherds, Ludmil Katzarkov, Tony Pantev, and N.~I. Shepherd-Barron.
\newblock Families of {$K3$} surfaces.
\newblock {\em J. Algebraic Geom.}, 7(1):183--193, 1998.

\bibitem[BM17]{bergeron}
N.~{Bergeron} and C.~{Matheus}.
\newblock {On special Lagrangian fibrations in generic twistor families of K3
  surfaces}.
\newblock {\em ArXiv e-prints}, March 2017.

\bibitem[Bor72]{borelmetric}
Armand Borel.
\newblock Some metric properties of arithmetic quotients of symmetric spaces
  and an extension theorem.
\newblock {\em J. Differential Geometry}, 6:543--560, 1972.
\newblock Collection of articles dedicated to S. S. Chern and D. C. Spencer on
  their sixtieth birthdays.

\bibitem[Bor98]{borcherds}
Richard~E. Borcherds.
\newblock Automorphic forms with singularities on {G}rassmannians.
\newblock {\em Invent. Math.}, 132(3):491--562, 1998.

\bibitem[Bor99]{borcherdszagier}
Richard~E. Borcherds.
\newblock The {G}ross-{K}ohnen-{Z}agier theorem in higher dimensions.
\newblock {\em Duke Math. J.}, 97(2):219--233, 1999.

\bibitem[Bri83]{brieskorn}
E.~Brieskorn.
\newblock {\em Die {M}ilnorgitter der exzeptionellen unimodularen
  {S}ingularit\"aten}, volume 150 of {\em Bonner Mathematische Schriften [Bonn
  Mathematical Publications]}.
\newblock Universit\"at Bonn, Mathematisches Institut, Bonn, 1983.

\bibitem[Bru02]{bruinier}
Jan~H. Bruinier.
\newblock {\em Borcherds products on {O}(2, {$l$}) and {C}hern classes of
  {H}eegner divisors}, volume 1780 of {\em Lecture Notes in Mathematics}.
\newblock Springer-Verlag, Berlin, 2002.

\bibitem[Cha18]{charles1}
François Charles.
\newblock Exceptional isogenies between reductions of pairs of elliptic curves.
\newblock {\em Duke Math. J.}, 167(11):2039--2072, 08 2018.

\bibitem[CMSP03]{carlson}
J.~Carlson, S.~M{\"u}ller-Stach, and C.~Peters.
\newblock {\em Period Mappings and Period Domains}.
\newblock Cambridge Studies in Advanced Mathematics. Cambridge University
  Press, 2003.

\bibitem[CU05]{clozelullmo}
Laurent Clozel and Emmanuel Ullmo.
\newblock Équidistribution de sous-vari\'et\'es sp\'eciales.
\newblock {\em Ann. of Math. (2)}, 161(3):1571--1588, 2005.

\bibitem[Del72]{deligne}
Pierre Deligne.
\newblock La conjecture de {W}eil pour les surfaces {$K3$}.
\newblock {\em Invent. Math.}, 15:206--226, 1972.

\bibitem[EO06]{eskinoh}
Alex Eskin and Hee Oh.
\newblock Representations of integers by an invariant polynomial and unipotent
  flows.
\newblock {\em Duke Math. J.}, 135(3):481--506, 2006.

\bibitem[ERS91]{eskinsarnak}
Alex Eskin, Ze\'ev Rudnick, and Peter Sarnak.
\newblock A proof of {S}iegel's weight formula.
\newblock {\em Internat. Math. Res. Notices}, (5):65--69, 1991.

\bibitem[{Fil}16]{filip}
S.~{Filip}.
\newblock {Counting special Lagrangian fibrations in twistor families of K3
  surfaces}.
\newblock {\em ArXiv e-prints}, December 2016.

\bibitem[GHS10]{gritsenko1}
V.~Gritsenko, K.~Hulek, and G.~K. Sankaran.
\newblock Moduli spaces of irreducible symplectic manifolds.
\newblock {\em Compos. Math.}, 146(2):404--434, 2010.

\bibitem[GHS13]{gritsenko}
V.~Gritsenko, K.~Hulek, and G.~K. Sankaran.
\newblock Moduli of {K}3 surfaces and irreducible symplectic manifolds.
\newblock In {\em Handbook of moduli. {V}ol. {I}}, volume~24 of {\em Adv. Lect.
  Math. (ALM)}, pages 459--526. Int. Press, Somerville, MA, 2013.

\bibitem[GP02]{goresky}
Mark Goresky and William Pardon.
\newblock Chern classes of automorphic vector bundles.
\newblock {\em Invent. Math.}, 147(3):561--612, 2002.

\bibitem[Gri74]{topics}
Phillip Griffiths.
\newblock {\em Topics in algebraic and analytic geometry}.
\newblock Princeton University Press, Princeton, N.J.; University of Tokyo
  Press, Tokyo, 1974.
\newblock Written and revised by John Adams, Mathematical Notes, No. 13.

\bibitem[Hel64]{helgason}
S.~Helgason.
\newblock {\em Differential Geometry and Symmetric Spaces}.
\newblock Pure and applied mathematics. Academic Press, 1964.

\bibitem[Huy03]{hk}
Daniel Huybrechts.
\newblock Compact hyperk\"ahler manifolds.
\newblock In {\em Calabi-{Y}au manifolds and related geometries
  ({N}ordfjordeid, 2001)}, Universitext, pages 161--225. Springer, Berlin,
  2003.

\bibitem[Huy16]{huybrechts}
Daniel Huybrechts.
\newblock {\em Lectures on {K}3 surfaces}, volume 158 of {\em Cambridge Studies
  in Advanced Mathematics}.
\newblock Cambridge University Press, Cambridge, 2016.

\bibitem[IRR14]{imam}
\"Ozlem Imamoglu, Martin Raum, and Olav~K. Richter.
\newblock Holomorphic projections and {R}amanujan's mock theta functions.
\newblock {\em Proc. Natl. Acad. Sci. USA}, 111(11):3961--3967, 2014.

\bibitem[Iwa97]{iwaniec}
Henryk Iwaniec.
\newblock {\em Topics in classical automorphic forms}, volume~17 of {\em
  Graduate Studies in Mathematics}.
\newblock American Mathematical Society, Providence, RI, 1997.

\bibitem[Mau14]{maulik}
Davesh Maulik.
\newblock Supersingular {K}3 surfaces for large primes.
\newblock {\em Duke Math. J.}, 163(13):2357--2425, 2014.
\newblock With an appendix by Andrew Snowden.

\bibitem[McG03]{mcgraw}
William~J. McGraw.
\newblock The rationality of vector valued modular forms associated with the
  {W}eil representation.
\newblock {\em Math. Ann.}, 326(1):105--122, 2003.

\bibitem[MR05]{royer}
Fran\c{c}ois Martin and Emmanuel Royer.
\newblock Formes modulaires et p\'eriodes.
\newblock In {\em Formes modulaires et transcendance}, volume~12 of {\em
  S\'emin. Congr.}, pages 1--117. Soc. Math. France, Paris, 2005.

\bibitem[Ogu03]{oguiso}
Keiji Oguiso.
\newblock Local families of {$K3$} surfaces and applications.
\newblock {\em J. Algebraic Geom.}, 12(3):405--433, 2003.

\bibitem[Oh04]{oh}
Hee Oh.
\newblock Hardy-{L}ittlewood system and representations of integers by an
  invariant polynomial.
\newblock {\em Geom. Funct. Anal.}, 14(4):791--809, 2004.

\bibitem[Pet15]{peterson}
Arie Peterson.
\newblock {\em {Modular forms on the moduli space of polarised K3 surfaces}}.
\newblock PhD thesis, {Korteweg-de Vries Institute for Mathematics}, June 2015.

\bibitem[Sar90]{sarnak}
Peter Sarnak.
\newblock {\em Some applications of modular forms}, volume~99 of {\em Cambridge
  Tracts in Mathematics}.
\newblock Cambridge University Press, Cambridge, 1990.

\bibitem[Sch73]{schmid}
Wilfried Schmid.
\newblock Variation of hodge structure: The singularities of the period
  mapping.
\newblock {\em Inventiones mathematicae}, 22:211--320, 1973.

\bibitem[Ser77]{cours}
Jean-Pierre Serre.
\newblock {\em Cours d'arithm\'etique}.
\newblock Presses Universitaires de France, Paris, 1977.
\newblock Deuxi\`eme \'edition revue et corrig\'ee, Le Math\'ematicien, No. 2.

\bibitem[Sie35]{siegeluber}
Carl~Ludwig Siegel.
\newblock Über die analytische {T}heorie der quadratischen {F}ormen.
\newblock {\em Ann. of Math. (2)}, 36(3):527--606, 1935.

\bibitem[ST17]{tang}
A.~N. {Shankar} and Y.~{Tang}.
\newblock {Exceptional splitting of reductions of abelian surfaces}.
\newblock {\em ArXiv e-prints}, June 2017.

\bibitem[Vau97]{vaughan}
R.~C. Vaughan.
\newblock {\em The {H}ardy-{L}ittlewood method}, volume 125 of {\em Cambridge
  Tracts in Mathematics}.
\newblock Cambridge University Press, Cambridge, second edition, 1997.

\bibitem[Voi02]{voisin}
C.~Voisin.
\newblock {\em Th{\'e}orie de Hodge et g{\'e}om{\'e}trie alg{\'e}brique
  complexe}.
\newblock Collection SMF. Soci{\'e}t{\'e} Math{\'e}matique de France, 2002.

\bibitem[Yaf07]{andreoort}
Andrei Yafaev.
\newblock The {A}ndr\'e-{O}ort conjecture---a survey.
\newblock In {\em {$L$}-functions and {G}alois representations}, volume 320 of
  {\em London Math. Soc. Lecture Note Ser.}, pages 381--406. Cambridge Univ.
  Press, Cambridge, 2007.

\bibitem[Zuc81]{zucker}
Steven Zucker.
\newblock Locally homogeneous variations of {H}odge structure.
\newblock {\em Enseign. Math. (2)}, 27(3-4):243--276, 1981.

\end{thebibliography}

%\subsection{Lagrangian fibrations}

%\begin{definition}
%A holomorphic Lagrangian fibration on a hyperkähler manifold is a surjective holomorphic map $X\rightarrow M$ where $\dim(X)=2\dim(M)$ and such that the restriction of the symplectic form of $X$ to the fibers of $\pi$ vanishes.
%\end{definition}

%Let $X\rightarrow M$ be a holomorphic Lagrangian fibration on a hyperkähler manifold $X$ and $\omega_{M}$ a Kähler class on $M$, then  $\pi^{*}\omega_{M}$ is a  parabolic nef class on $X$. Conversely, a theorem of Matsushita \cite{matsushita} states that a parabolic semi-ample line bundle defines a holomorphic Lagrangian fibration. We call such a line bundle a {\it Lagrangian class}.  
%\begin{conjecture}[Hyperkähler SYZ conjecture]
%Any non-trivial parabolic nef line bundle on a hyperkähler manifold is Lagrangian. 
%\end{conjecture}
\end{document}